\documentclass{article}
\usepackage{latexsym,amsmath,amssymb,amsthm,enumerate}
\usepackage{color}
\usepackage[pdftex,colorlinks=true,bookmarks=false,citecolor=blue]{hyperref}
\usepackage{mathrsfs}
\usepackage{dsfont}
\usepackage{multirow}
\usepackage{booktabs}
\usepackage{rotating} 
\usepackage{graphicx,xcolor,float}
\usepackage{subfigure}
\usepackage{algorithm}
\usepackage{algorithmic}

\allowdisplaybreaks

\newcommand{\inner}[2]{\langle#1,#2\rangle}
\newcommand{\norm}[1]{\|#1\|}

\theoremstyle{definition}
\newtheorem{theorem}{Theorem}[section]
\newtheorem{lemma}[theorem]{Lemma}
\newtheorem{proposition}[theorem]{Proposition}

\newtheorem{definition}[theorem]{Definition}
\newtheorem{remark}[theorem]{Remark}
\newtheorem{example}{Example}

\makeatletter
\renewcommand{\@seccntformat}[1]{\csname the#1\endcsname.\hspace{1em}}
\makeatother

\makeatletter
\def\thanks#1{\protected@xdef\@thanks{\@thanks
          \protect\footnotetext{#1}}}
\makeatother

\allowdisplaybreaks

\begin{document}

\title{Strong Convergence of Relaxed Inertial Inexact Progressive Hedging
Algorithm for Multi-stage Stochastic Variational Inequality Problems
\thanks{This research is supported by National Natural Science Foundation of China (Grant No.12071324)}}
\author{Jiaxin CHEN \textsuperscript{1,}\textsuperscript{2} \and Zunjie HUANG \textsuperscript{1} \and Haisen ZHANG \textsuperscript{1} \thanks{\textsuperscript{1} School of Mathematical Sciences, Sichuan Normal University, Sichuan $610066$, China.} \thanks{\textsuperscript{2} Chengdu Public Security Bureau, Sichuan $610017$, China.}  \thanks{Corresponding author: Haisen Zhang. E-mail address: haisenzhang@yeah.net}}
\date{ }
\maketitle

\begin{abstract}
A Halpern-type relaxed inertial inexact progressive hedging algorithm ($\mbox{PHA}$) is proposed for solving multi-stage stochastic variational inequalities in general probability spaces. The subproblems in this algorithm are allowed to be  calculated inexactly. It is found that the Halpern-type relaxed inertial inexact $\mbox{PHA}$ is closely related to the Halpern-type relaxed inertial inexact proximal point algorithm ($\mbox{PPA}$). The strong convergence of the Halpern-type relaxed inertial inexact $\mbox{PHA}$ is proved under appropriate conditions. Some numerical examples are given to indicate that the over-relaxed parameter and the inertial term can accelerate the convergence of the algorithm.
\end{abstract}

\textbf{Key words:} Multi-stage stochastic variational inequality; inexact progressive hedging algorithm; proximal point algorithm; nonanticipativity; strong convergence.

\vspace{+0.5em}

\textbf{AMS subject classifications:} 65K10; 65K15; 90C15; 90C25

\section{Introduction}

The stochastic variational inequalities, as useful tools for solving optimization problems,  equilibrium problems and  optimal control problems with uncertainty, have important applications in economics, management, finance and engineering. In order to model different stochastic problems, different types of stochastic variational inequalities ($\mbox{SVIs}$ for short) have been proposed over the past two decades. For instance, the one-stage Expected-Value $\mbox{SVIs}$ \cite{Gurkan99,Jiang16,Ravat11,Ravat17}, the one-stage Expected-Residual-Minimization $\mbox{SVIs}$ \cite{Chen05,Ma13,Zhao17}, and the two-stage $\mbox{SVIs}$ \cite{Chen17,Chen19,Chen2019,JiangChenChen20,Jiang21,Li20,Sun21,WangChen23}.
The multi-stage $\mbox{SVIs}$ with nonanticipativity constraints in discrete probability spaces were first proposed by Rockafellar and Wets \cite{Rockafellar17} in 2017. The multi-stage $\mbox{SVIs}$ in general probability spaces were introduced in \cite{Zhang22}.   Unlike the discrete case, the multi-stage $\mbox{SVIs}$ defined in general probability spaces are typically a class of infinite-dimensional variational inequality problems. As illustrated in \cite[Example 2.2]{Zhang22},  extending the concept of multi-stage $\mbox{SVIs}$ to general probability spaces is primarily motivated by the need to solve infinite-dimensional stochastic convex optimal control problems.

It is shown in \cite{Rockafellar17} that  the one-stage Expected-Value $\mbox{SVIs}$ and the two-stage $\mbox{SVIs}$ are two special cases of the multi-stage $\mbox{SVIs}$ with nonanticipativity constraints. Additionally, there are essential differences between the multi-stage $\mbox{SVIs}$ and the one-stage Expected-Value $\mbox{SVIs}$ and the two-stage $\mbox{SVIs}$. On the one hand, the decision variable in the one-stage Expected-Value $\mbox{SVIs}$ is deterministic (independent of sample points) and can thus be viewed as structured variational inequalities defined in a finite-dimensional space. In contrast, the decision variables in the two-stage and multi-stage $\mbox{SVIs}$ are random variables (depending on sample points). Specifically, when the number of sample points in the sample space is infinite, the two-stage and multi-stage $\mbox{SVIs}$ are infinite-dimensional variational inequalities defined on a proper space of random variables. On the other hand, the first stage of the two-stage $\mbox{SVIs}$ requires the decision variable to be independent of random information, while the second stage allows the decision variable to utilize all the random information. Consequently, in two-stage $\mbox{SVIs}$, computing the projection onto the nonanticipative constraint set only requires calculating the mathematical expectation of the first stage. In multi-stage $\mbox{SVIs}$, from the second stage to the penultimate stage, the decision variable can only utilize partial random information. Therefore, in multi-stage $\mbox{SVIs}$, calculating the projection onto the nonanticipative constraint set requires computing a collection of conditional mathematical expectations. Compared with the one-stage Expected-Value $\mbox{SVIs}$ and two-stage $\mbox{SVIs}$, research on algorithms for multi-stage $\mbox{SVIs}$ is still in its initial stages.

The progressive hedging algorithm ($\mbox{PHA}$ for short) is one of the main methods for solving multi-stage $\mbox{SVIs}$. It was first proposed by Rockafellar and Wets in \cite{Rockafellar91} for solving stochastic convex programming problems with nonanticipativity constraints and then extended by Rockafellar and Sun in \cite{Rockafellar19} to solve two-stage and multi-stage $\mbox{SVIs}$. Subsequently, the $\mbox{PHA}$ was developed to solve two-stage quadratic games under uncertainty (see Zhang, Sun and Xu \cite{Zhang19}), Lagrangian multi-stage monotone $\mbox{SVIs}$ (see Rockafellar and Sun \cite{Rockafellar20}), multi-stage stochastic minimization problems (see Sun, Xu and Zhang \cite{Sun20}), multi-stage pseudomonotone $\mbox{SVIs}$ (see Cui, Sun and Zhang \cite{Cui23}) and nonmonotone $\mbox{SVIs}$ (see Cui and Zhang \cite{CZ23}). It was proved in \cite{Rockafellar19}  that
the $\mbox{PHA}$ for multi-stage $\mbox{SVIs}$ is equivalent to the proximal point algorithm ($\mbox{PPA}$ for short) for finding zeros of a proper maximal monotone operator.

The $\mbox{PPA}$ is a powerful algorithm for solving optimization problems and maximal monotone inclusion problems. It and its various types of modifications have been extensively studied in the past decades. The weak convergence of $\mbox{PPA}$ was first proved by Martinet \cite{Martinet70} in Hilbert space in the special case that the regularization parameter is a fixed constant. It should be noted that the $\mbox{PPA}$ is an implicit iterative algorithm. To obtain the next iteration, it is imperative to solve an optimization subproblem  or a monotone inclusion subproblem. Usually, accurate solutions to those subproblems are difficult to achieve. To address this challenge, Rockafellar~\cite{Rockafellar76} proposed an inexact version of $\mbox{PPA}$ with variable parameters for solving monotone inclusion problems, which allows the proximal subproblem in each iteration to be solved inexactly. When the regularization parameter sequence is bounded away from zero, Rockafellar~\cite{Rockafellar76} proved the weak convergence and linear convergence rate of inexact $\mbox{PPA}$ in Hilbert space.
In general, the inexact $\mbox{PPA}$ does not have strong convergence in infinite-dimensional spaces (see, for instance, G{\"{u}}ler \cite{Guler91}). To obtain the strong convergence of the inexact $\mbox{PPA}$, many modified inexact $\mbox{PPAs}$ have been proposed. In \cite{Halpern67}, Halpern found that a suitable convex combination of the initial point and the Picard iteration of the nonexpansive mapping can make the Picard iteration sequence converge strongly. By combining Halpern's method with Rockafellar's inexact $\mbox{PPA}$, Xu \cite{Xu02} proposed a Halpern-type inexact $\mbox{PPA}$ and obtained the strong convergence of the algorithm.

In order to accelerate the convergence of $\mbox{PPAs}$, many  accelerated $\mbox{PPAs}$ have been proposed in recent years. One classical accelerated $\mbox{PPA}$ is the relaxed inexact $\mbox{PPA}$, which was first studied by Gol'shtein and Tret'yakov \cite{Golshtein79}. The global convergence of the relaxed inexact $\mbox{PPA}$ was obtained in the finite-dimensional Euclidean space. The regularization parameter is a fixed positive constant.  Eckstein and Bertsekas \cite{Eckstein92} proved the weak convergence of the relaxed inexact $\mbox{PPA}$ with variable parameters in  Hilbert space. Based on the techniques in \cite{Eckstein92} and \cite{Xu02}, Wang and Cui \cite{Wang15} proposed a strongly convergent Halpern-type relaxed inexact $\mbox{PPA}$ in Hilbert space. As shown in \cite{Bertsekas82,Eckstein92,Wang15}, the over-relaxed parameters can speed up the convergence of the relaxed inexact $\mbox{PPAs}$.

Another type of accelerated $\mbox{PPA}$ is called the inertial $\mbox{PPA}$, wherein the next iterate is defined by utilizing the previous two iterates, as seen in \cite{Alvarez01,Lorenz15,Mainge08}. The key idea of inertial $\mbox{PPA}$ originally came from the multistep method for solving nonlinear equations proposed by Polyak \cite{Polyak64}. The weak convergence of inertial $\mbox{PPA}$ was first obtained by Alvarez and Attouch in \cite{Alvarez01} in Hilbert space. Subsequently, Alvarez  \cite{Alvarez04} proposed a relaxed inertial hybrid projection $\mbox{PPA}$ that unified the relaxed and inertial acceleration strategies of $\mbox{PPA}$.  The weak convergence of the algorithm is proved in Hilbert space under certain implementable inexact conditions. Recently, the convergence of the relaxed inertial proximal algorithm has been further studied in \cite{Attouch2020} in Hilbert space.
Tan, Zhou and Li \cite{Tan20} proposed two modified inertial Mann algorithms for nonexpansive mapping and proved the strong convergence of these two algorithms. These two algorithms can be regarded as the modified relaxed inertial $\mbox{PPA}$ with under-relaxed parameters.

The equivalence between the $\mbox{PHA}$ for multi-stage $\mbox{SVIs}$ and the $\mbox{PPA}$ for maximal monotone inclusion problems established in \cite{Rockafellar19} motivates us to propose some accelerated $\mbox{PHAs}$ for  multi-stage $\mbox{SVIs}$ using the acceleration techniques developed in $\mbox{PPAs}$. Along this line, Chen and Zhang \cite{Chen23} proposed two relaxed inexact $\mbox{PHAs}$ for solving multi-stage $\mbox{SVIs}$ in general probability spaces. Similar to the relaxed inexact $\mbox{PPAs}$, the over-relaxed parameter can speed up the convergence of the algorithms.

In this paper, we shall propose a Halpern-type relaxed inertial inexact $\mbox{PHA}$ for solving multi-stage $\mbox{SVIs}$. It is proved that the  Halpern-type relaxed inertial inexact $\mbox{PHA}$ is equivalent to a Halpern-type relaxed inertial inexact $\mbox{PPA}$ for finding zeros of a
proper maximal monotone operator. Based on the equivalence, the strong convergence of the Halpern-type relaxed inertial inexact $\mbox{PHA}$ is proved in general probability spaces by the strong convergence result of the Halpern-type relaxed inertial inexact $\mbox{PPA}$.  Numerical examples show that the Halpern-type relaxed inertial inexact $\mbox{PHA}$ has better numerical performance than the original $\mbox{PHA}$ and Halpern-type relaxed inexact $\mbox{PHA}$.  It should be remarked that the convergence result of Halpern-type relaxed inertial inexact $\mbox{PPA}$ for maximal monotone operators presented in this paper is novel. On the one hand, different from \cite{Alvarez04} and \cite{Attouch2020}, in our algorithm the inexact calculation of the resolvent and Halpern's technique for strong  convergence of the algorithm are considered. On the other hand, compared with \cite{Tan20}, our algorithm involves the inexact calculation of the resolvent and the over-relaxed parameter.

The rest of this paper is organized as follows: In Section 2, we give some basic notations, definitions, and results from set-valued and variational analysis, and introduce the models of multi-stage $\mbox{SVIs}$ in a general probability space. In Section 3,  we study the Halpern-type relaxed inertial inexact $\mbox{PHA}$ for multi-stage $\mbox{SVIs}$ and its equivalence with the Halpern-type relaxed inertial inexact $\mbox{PPA}$.  Based on  the equivalence, we prove the strong convergence of the Halpern-type relaxed inertial inexact $\mbox{PHA}$. In Section 4, numerical examples are presented to show the efficiency of the Halpern-type relaxed inertial inexact $\mbox{PHA}$.  Finally, in Section 5, we give a proof of the strong convergence of the Halpern-type relaxed inertial inexact $\mbox{PPA}$ for maximal monotone inclusion problems.

\section{Preliminaries}
In this section, we first recall some basic notions and results in set-valued and variational analysis. Then, we give the definition of multi-stage $\mbox{SVIs}$ in a general probability space.

\subsection{Set-valued and Variational Analysis}

In this subsection, we recall some notations and preliminary results in set-valued and variational analysis. We refer the readers to \cite{Aubin09} and \cite{Bauschke01} for more details.

Let $\mathcal{H}$ be a Hilbert space and $\mathcal{D}$ be a nonempty closed convex subset of $\mathcal{H}$. The symbols $\norm{\cdot}_{\mathcal{H}}$ and $\inner{\cdot}{\cdot}_{\mathcal{H}}$ are used to denote the norm and inner product of $\mathcal{H}$, respectively.
Let $\{u^{k}\}_{k=0}^{\infty}$ be a sequence in $\mathcal{H}$, we write $u^{k}\rightarrow u$ and $u^{k}\rightharpoonup u$ to indicate that the sequence $\{u^{k}\}_{k=0}^{\infty}$ converges strongly and weakly to $u$, respectively. A mapping $F:\mathcal{D}\rightarrow\mathcal{D}$ is said to be nonexpansive if $$\norm{F(x)-F(y)}_{\mathcal{H}}\leq\norm{x-y}_{\mathcal{H}}$$ for all $x,y\in\mathcal{D}$.
The normal cone $N_{\mathcal{D}}(x)$ of $\mathcal{D}$ at $x$ is defined by
\begin{equation*}
 N_{\mathcal{D}}(x):=
\begin{cases}
\big\{\eta\in \mathcal{H}~\big|~\inner{\eta}{y-x}_{\mathcal{H}}\leq 0,~\forall~y\in \mathcal{D} \big\}, &  x\in \mathcal{D};  \vspace{+0.5em}\\
 \emptyset, &  \mbox{otherwise}.
\end{cases}
\end{equation*}
The metric projection of $x\in \mathcal{H}$ onto $\mathcal{D}$ is defined by
\begin{equation*}
\Pi_{\mathcal{D}}(x):=\big\{\bar{x}\in \mathcal{D}~\big|~\norm{x-\bar{x}}_{\mathcal{H}}=\inf\limits_{y\in \mathcal{D}}\norm{x-y}_{\mathcal{H}}\big\}.
\end{equation*}
Clearly, the operator $\Pi_{\mathcal{D}}: \mathcal{H}\to \mathcal{D}$ is  single-valued  and nonexpansive. In addition,
$$\inner{x-\Pi_{\mathcal{D}}(x)}{y-\Pi_{\mathcal{D}}(x)}_{\mathcal{H}}\leq 0$$
for any  $y\in \mathcal{D}$, or equivalently, $x-\Pi_{\mathcal{D}}(x)\in  N_{\mathcal{D}}(\Pi_{\mathcal{D}}(x))$.

Let $\mathcal{H}_{1}$ be a nonempty closed linear subspace of $\mathcal{H}$. The orthogonal projection of $x\in \mathcal{H}$ onto $\mathcal{H}_{1}$ is defined by
\begin{equation*}
P_{\mathcal{H}_{1}}(x):=\big\{x_1\in \mathcal{H}_{1}~\big|~\inner{x-x_1}{y}_{\mathcal{H}}=0,~\forall~y\in \mathcal{H}_{1}\big\}.
\end{equation*}
It is clear that $P_{\mathcal{H}_{1}}$ is a bounded linear operator.

Let $(X,\mathscr{S},\mu)$ be a complete $\sigma$-finite measure space and $Y$ be a complete separable metric space. A set-valued mapping $T:X\rightsquigarrow Y$ is characterized by its graph $\mbox{Gph}(T)$, the subset of the product space $X\times Y$ defined by $\mbox{Gph}(T):=\big\{(x,y)\in X\times Y~\big|~y\in T(x)\big\}$. We will use $T(x)$ to denote the value of $T$ at $x$, which is a subset of $Y$. The domain of $T$ is denoted by $\mbox{Dom}(T)$ and defined by $\mbox{Dom}(T):=\big\{x\in X~\big|~ T(x)\neq \emptyset \big\}$. The image of $T$ is denoted by $\mbox{Im}(T)$ and defined by $\mbox{Im}(T):=\bigcup\limits_{x\in X}T(x)$. The inverse of $T$, denoted by $T^{-1}$, is defined through its graph $\mbox{Gph}(T^{-1}):=\big\{(y,x)\in Y\times X~\big|~(x,y)\in \mbox{Gph}(T)\big\}$.

\begin{definition}\label{monotone}\cite[Definition 3.5.1, Definition 3.5.4]{Aubin09}
A set-valued mapping $T:\mathcal{H}\rightsquigarrow \mathcal{H}$ is said to be monotone if its graph is monotone in the sense that
\begin{equation*}
\inner{\xi-\zeta}{u-v}_{\mathcal{H}}\geq 0,~\forall~(u,\xi),(v,\zeta)\in \mbox{Gph}(T).
\end{equation*}
A set-valued mapping $T$ is said to be maximal monotone if it is monotone and there is no other monotone set-valued mapping whose graph strictly contains the graph of $T$.
\end{definition}

\begin{lemma} \cite[Lemma 3.4]{Chen23}\label{maxi ATA}
Let $A: \mathcal{H}\to \mathcal{H}$ be a symmetric invertible linear operator, $T:\mathcal{H}\rightsquigarrow \mathcal{H}$ be a maximal monotone set-valued map. Then, the set-valued map $ATA:\mathcal{H}\rightsquigarrow \mathcal{H}$ is also maximal monotone.
\end{lemma}

A single-valued map $F:\mathcal{H}\rightarrow \mathcal{H}$ is hemicontinuous if it is continuous from line segments in $\mathcal{H}$ to the weak topology in $\mathcal{H}$, i.e., for any $x,y\in\mathcal{H}$ and $z\in \mathcal{H}$, the function
\begin{equation*}
t\mapsto \inner{z}{F(tx+(1-t)y)}_{\mathcal{H}}, \quad t\in [0,1]
\end{equation*}
is continuous.

\begin{lemma}\cite[Theorem 3]{Rockafellar70}\label{F and Nk} Let  $F:\mathcal{H}\rightarrow \mathcal{H}$ be a single-valued mapping. If $F$ is monotone and hemicontinuous, then the set-valued mapping $F+N_{\mathcal{D}}:\mathcal{H}\rightsquigarrow \mathcal{H}$ is maximal monotone.
\end{lemma}

Let $T:\mathcal{H}\rightsquigarrow \mathcal{H}$ be a  maximal monotone operator.  The resolvent  $J_{r}^{T}$ of $T$  with parameter $r>0$ is defined by
\begin{equation*}
J_{r}^{T}(u)=(I+rT)^{-1}(u),~~\forall~u\in\mathcal{H}.
\end{equation*}
The resolvent $J_{r}^{T}$ has the following basic properties, see \cite[Corollary 23.11]{Bauschke01}, \cite[Lemma 3.3]{Marino04} and \cite[Proposition 1]{Rockafellar76}.

\begin{proposition}\label{Convergence proposition} Let $T:\mathcal{H}\rightsquigarrow \mathcal{H}$ be a maximal monotone operator and $J_{r}^{T}$ be its resolvent with parameter $r>0$. Then
\begin{enumerate}[{\rm (i)}]
          \item  $J_{r}^{T}:\mathcal{H}\rightarrow\mathcal{H}$ is a single-valued  nonexpansive mapping;
          \item  $\norm{u-J_{r}^{T}(u)}_{\mathcal{H}}\leq 2\norm{u-J_{r'}^{T}(u)}_{\mathcal{H}}$ for all $0<r\leq r'$ and $u\in \mathcal{H}$;
          \item  $r^{-1}(I-J_{r}^{T})(u)\in T(J_{r}^{T}(u))$ for all $u\in \mathcal{H}$;
          \item  $\inner{J_{r}^{T}(u)-J_{r}^{T}(v)}{(I-J_{r}^{T})(u)-(I-J_{r}^{T})(v)}_{\mathcal{H}}\geq 0$ for all $u,v\in \mathcal{H}$;
          \item  $\norm{J_{r}^{T}(u)-J_{r}^{T}(v)}_{\mathcal{H}}^{2}
+\norm{(I-J_{r}^{T})(u)-(I-J_{r}^{T})(v)}_{\mathcal{H}}^{2}\leq \norm{u-v}_{\mathcal{H}}^{2}$ for all $u,v\in \mathcal{H}$;
          \item $\norm{(2J_{r}^{T}-I)(u)-(2J_{r}^{T}-I)(v)}_{\mathcal{H}}\leq \norm{u-v}_{\mathcal{H}}$ for all $u,v\in \mathcal{H}$.
        \end{enumerate}
\end{proposition}

Clearly, $0\in T(u)$  if and only if $u$ is a fixed point of $J_{r}^{T}$. By the properties of the maximal monotone operator, it is easy to check that the set of zeros $T^{-1}(0)$ for a maximal monotone operator $T$ is a closed convex subset of $\mathcal{H}$ (see, i.g., \cite[Proposition 23.39]{Bauschke01}).

Let $\mathcal{H}_{1}$ and $\mathcal{H}_{2}$ be mutually complementary subspaces of $\mathcal{H}$, i.e., $\mathcal{H}_{1}=\mathcal{H}_{2}^{\bot}$ and $\mathcal{H}_{2}=\mathcal{H}_{1}^{\bot}$.

\begin{definition}\cite{Spingarn1983}\label{the partial inverse}
Let  $T:\mathcal{H}\rightsquigarrow \mathcal{H}$ be a monotone operator. The set
\begin{equation}\label{partial inverse}
\big\{(P_{\mathcal{H}_{1}}(u)+P_{\mathcal{H}_{2}}(\xi),P_{\mathcal{H}_{1}}(\xi)+P_{\mathcal{H}_{2}}(u))\in \mathcal{H}\times \mathcal{H}~\big|~(u,\xi)\in\mathcal{H}\times\mathcal{H}, ~\xi\in T(u)\big\}
\end{equation}
is the graph of a monotone operator. The monotone operator whose graph is described in this manner is denoted by $T_{\mathcal{H}_{1}}^{-1}$ and is called the partial inverse of $T$ with respect to $\mathcal{H}_{1}$.
\end{definition}

It follows from the definition of partial inverse that, the partial inverse
of $T$ with respect to $\mathcal{H}_{1}=\{0\}$ is the ordinary inverse of $T$ and the partial inverse with respect to $\mathcal{H}_{1}=\mathcal{H}$ is just $T$ itself. In addition, for any closed subspace $\mathcal{H}_{1}$, the double  partial inverse of $T$ with respect to $\mathcal{H}_{1}=\{0\}$ is $T$ itself, i.e., $(T _{\mathcal{H}_{1}})_{\mathcal{H}_{1}}=T$.

\begin{lemma}\cite[Proposition 2.1]{Spingarn1983}\label{TN is maximal monotone}
The set-valued mapping $T_{\mathcal{H}_{1}}^{-1}:\mathcal{H}\rightsquigarrow \mathcal{H}$ is (maximal) monotone if and only if the set-valued mapping $T:\mathcal{H}\rightsquigarrow \mathcal{H}$ is (maximal) monotone.
\end{lemma}

To end this subsection, we list the following three elementary results.

\begin{lemma} \cite[Corollary 4.28]{Bauschke01}  \label{weakly-strongly closed}
Let $\mathcal{D}$ be a nonempty closed convex subset of $\mathcal{H}$,  $F:\mathcal{D}\rightarrow\mathcal{H}$ be a nonexpansive mapping,  $\{u^{k}\}_{k=0}^{\infty}$ be a sequence in $\mathcal{D}$ and,  $u$ be a point in $\mathcal{H}$. Suppose that $u^{k}\rightharpoonup u$ and $u^{k}-F(u^{k})\rightarrow 0$. Then $u\in Fix(F)$, where $Fix(F):=\{u\in\mathcal{H}~|~F(u)=u\}$.
\end{lemma}

\begin{lemma}\label{Norm proposition}\cite[Corollary 2.15]{Bauschke01}
Let $x,y\in\mathcal{H}$ and $t\in \mathbb{R}$. Then
\begin{enumerate}[{\rm (i)}]
\item $\norm{x+y}_{\mathcal{H}}^{2}\leq\norm{x}_{\mathcal{H}}^{2}+2\inner{y}{x+y}_{\mathcal{H}}$;
\item $\norm{tx+(1-t)y}_{\mathcal{H}}^{2}=t\norm{x}_{\mathcal{H}}^{2}+(1-t)\norm{y}_{\mathcal{H}}^{2}-t(1-t)\norm{x-y}_{\mathcal{H}}^{2}$.
\end{enumerate}
\end{lemma}

\begin{lemma}\cite[Lemma 2.5]{Xu02}\label{Strong convergence}
Let $\{t_{k}\}_{k=0}^{\infty}$ be a sequence of non-negative real numbers such that
\begin{equation*}
t_{k+1}\leq (1-\alpha_{k})t_{k}+\alpha_{k}\gamma_{k}+\rho_{k},~~k\geq 0,
\end{equation*}
where $\{\alpha_{k}\}_{k=0}^{\infty}$, $\{\gamma_{k}\}_{k=0}^{\infty}$ and $\{\rho_{k}\}_{k=0}^{\infty}$ satisfy
\begin{enumerate}[{\rm (i)}]
  \item $\{\alpha_{k}\}_{k=0}^{\infty}\subseteq [0,1]$, $\sum\limits_{k=0}^{\infty}\alpha_{k}=\infty$, or equivalently, $\prod\limits_{k=0}^{\infty}(1-\alpha_{k})=0$;
  \item $\limsup\limits_{k\rightarrow \infty}\gamma_{k}\leq 0$;
  \item $\rho_{k}\geq 0(k\geq 0)$, $\sum\limits_{k=0}^{\infty}\rho_{k}<\infty$.
\end{enumerate}
Then $\lim\limits_{k\rightarrow \infty}t_{k}=0$.
\end{lemma}

\subsection{Multi-stage $\mbox{SVIs}$ in General Probability Spaces}

In this subsection, we first introduce some basic notations and concepts in probability. Then, we give the definition of multi-stage $\mbox{SVIs}$ in general probability spaces. We refer the readers to \cite{Rockafellar17} and \cite{Zhang22} for more details about the models of the Multi-stage $\mbox{SVIs}$.

Let $(\Omega,\mathscr{F},\mathbb{P})$ be a complete probability space. Here $\Omega$ is a sample space, $\mathscr{F}$ is a $\sigma$-algebra on $\Omega$, and $\mathbb{P}$ is a probability measure defined on $(\Omega,\mathscr{F})$. Any element of $\Omega$, denoted by $\omega$, is called a sample point.  As usual, when the context is clear, we omit the $\omega$ ($\in \Omega$) argument in the defined mappings/functions. Denote by $\mathscr{O}$ the collection of all $\mathbb{P}$-null sets. We say that a property holds almost surely (a.s.) if there is a set $A\in \mathscr{O}$ such that the property holds for every $\omega\in \Omega\setminus A$. Let $R^{n}~(n\in \mathds{N})$ be an n-dimensional Euclidean space with Borel $\sigma$-algebra $\mathscr{B}(R^{n})$. The symbols $|\cdot|$ and $\inner{\cdot}{\cdot}$ are used to denote the norm and inner product in $R^{n}$, respectively. A mapping $\xi:\Omega\rightarrow R^{n}$ is called an $\mathscr{F}$-measurable random vector if
\begin{equation*}
\xi^{-1}(A):=\big\{\omega\in\Omega~\big|~\xi(\omega)\in A\big\}\in \mathscr{F},~\forall~A\in\mathscr{B}(R^{n}).
\end{equation*}
When a random vector $\xi$ is integrable with respect to the probability measure $\mathbb{P}$, we will use  $\mathds{E}~\xi=\int_{\Omega}\xi(\omega)\mathbb{P}(d\omega)$ to denote the expectation of $\xi$. Denote by $\mathcal{L}^2$ the Hilbert space of all the square-integrable random vectors taking values in $R^{n}$, i.e.,
\begin{equation}\label{L2definition}
\mathcal{L}^{2}:=\big\{\xi:\Omega\rightarrow R^{n}~\big|~\xi\mbox{ is $\mathscr{F}$-measurable and } \mathds{E}|\xi|^{2}< +\infty\big\}.
\end{equation}
For any random vectors $\xi$, $\eta\in \mathcal{L}^{2}$, the norm of $\xi$ is defined by $$\norm{\xi}_{\mathcal{L}^{2}}:= \big[\mathds{E}|\xi|^{2}\big]^{\frac{1}{2}}= \Bigg[\int_{\Omega}|\xi(\omega)|^{2}\mathbb{P}(d\omega)\Bigg]^{\frac{1}{2}},$$
and the inner product of $\xi$ and $\eta$ is defined by $$\inner{\xi}{\eta}_{\mathcal{L}^{2}}:= \mathds{E}\inner{\xi}{\eta}= \int_{\Omega}\inner{\xi(\omega)}{\eta(\omega)}\mathbb{P}(d\omega).$$

\begin{definition}\label{Conditional expectation}\cite[Definition 1, Page 213]{Shiryaev16}
Let $\xi:\Omega\rightarrow R^{n}$ be an $\mathscr{F}$-measurable random vector with $\mathds{E}|\xi|<+\infty$ and $\mathscr{G}$ be a sub-$\sigma$-algebra of $\mathscr{F}$. The conditional expectation of $\xi$ with respect to $\mathscr{G}$, denoted by $\mathds{E}[\xi|\mathscr{G}]$, is a random vector such that
\begin{enumerate}[{\rm (i)}]
\item $\mathds{E}[\xi|\mathscr{G}]:\Omega\rightarrow R^{n}$ is $\mathscr{G}$-measurable;
\item $\int_{A}\xi(\omega)\mathbb{P}(d\omega)=\int_{A}\mathds{E}[\xi|\mathscr{G}](\omega)\mathbb{P}(d\omega),~\forall~ A\in \mathscr{G}.$
\end{enumerate}
\end{definition}

Let $n_{0},n_{1},\ldots,n_{N-1}\in \mathds{N}$, $n_{0}+n_{1}+\cdots+n_{N-1}=n$, $\mathscr{F}_{i}, i=0,1,\ldots,N-1$ be a collection of sub-$\sigma$-algebras such that $\mathscr{O}\subset\mathscr{F}_{0}\subset\mathscr{F}_{1}\subset\cdots\subset\mathscr{F}_{N-1}\subset \mathscr{F}$. Define the nonanticipativity closed linear subspace $\mathcal{N}$ of $\mathcal{L}^{2}$ by
\begin{equation}\label{Nonanticipativity set2}
\mathcal{N}\!:=\!\big\{x=(x_{0},x_{1},\ldots,x_{N-1})\in \mathcal{L}^{2}\,\big|\,x_{i}:\Omega\rightarrow R^{n_{i}}~\mbox{is}  ~\mathscr{F}_{i}\mbox{-measurable},~\forall~ i\!=\!0,1,\ldots,N\!-\!1\big\}.
\end{equation}
By the properties of conditional expectation, the orthogonal projection of a random vector $x\in \mathcal{L}^{2}$  onto $\mathcal{N}$ can be represented by
\begin{equation}\label{orthogonal proj N}
P_{\mathcal{N}}(x)=\Big\{y=(y_{0},y_{1},...,y_{N-1})\in \mathcal{N} \ \Big|\ y_{i}=\mathds{E} \big[x_{i} | \mathscr{F}_{i}\big]  \mbox{ a.s.,}\ \forall\ i= 0,1,\ldots,N-1\Big\},
\end{equation}
and, the orthogonal complementary subspace of $\mathcal{N}$ in $\mathcal{L}^2$ is characterized  by
\begin{equation}\label{RockanonantispaceM}
\mathcal{M}=\Big\{y=(y_{0},y_{1},...,y_{N-1})\in \mathcal{L}^2 \ \Big| \  \mathds{E} \big[y_{i} | \mathscr{F}_{i}\big]=0 \ \mbox{a.s.,}\ \forall\ i= 0,1,\ldots,N-1\Big\}.
\end{equation}

Let $\mathcal{C}$ be a nonempty closed convex subset of $\mathcal{L}^2$ defined  by
\begin{equation}\label{constraintsetC}
\mathcal{C}:=\Big\{x=(x_{0},x_{1},...,x_{N-1})\in \mathcal{L}^2\ \Big|\ x_{i}(\omega)\in C_{i}(\omega)  \mbox{ a.s. } \omega\in\Omega,  \forall~i= 0,1,\ldots,N-1 \Big\},
\end{equation}
with  $C_{i}:\Omega\rightsquigarrow R^{n_{i}}$ being given $\mathscr{F_{i}}$-measurable set-valued mapping  with   nonempty closed convex values, $\ i= 0,1,\ldots,N-1$.

Define $C:\Omega\rightsquigarrow R^n$  by
$C(\omega)=C_{0}(\omega)\times C_{1}(\omega)\times\cdots\times C_{N-1}(\omega),  a.s. \ \omega\in \Omega$. It is clear that, $\mathcal{C}$ is the set of square-integrable selections of the set valued mapping $C:\Omega\rightsquigarrow R^n$ and assumed to be nonempty. Some mild conditions for the nonemptiness of such  square-integrable selection can be found in \cite{Chen23}.

Let $F: \mathcal{L}^2  \to \mathcal{L}^2$ be a given mapping. The multi-stage stochastic variational inequality discussed in this paper, denoted by $\mbox{MSVI}(F,\mathcal{C}\cap\mathcal{N})$,  is to find  $x^{\ast}\in\mathcal{C}\cap\mathcal{N}$ such that
\begin{equation}\label{SVI}
-F(x^{\ast})\in N_{\mathcal{C}\cap \mathcal{N}}(x^{\ast}),
\end{equation}
where $N_{\mathcal{C}\cap \mathcal{N}}(x^{\ast})$ is the normal cone of $\mathcal{C}\cap \mathcal{N}$ on $x^{\ast}$ in $\mathcal{L}^2$.

When the sum  rule
\begin{equation}\label{sum rule}
N_{\mathcal{C} \cap\mathcal{N}}(x)=N_{\mathcal{C}}(x)+N_{\mathcal{N}}(x), ~ \forall \ x\in\mathcal{C} \cap\mathcal{N}
\end{equation}
is satisfied, the multi-stage stochastic variational inequality $\mbox{MSVI}(F,\mathcal{C}\cap\mathcal{N})$ \eqref{SVI} is equivalent to the multi-stage stochastic variational inequality in extensive form: Find  $x^*\in\mathcal{C}$ and $y^*\in \mathcal{M}$ such that
\begin{equation}\label{SVIextensivedefini}
-F(x^*)-y^*\in N_{\mathcal{C}}(x^*).
\end{equation}

Clearly, any solution of the multi-stage stochastic variational inequality in extensive form \eqref{SVIextensivedefini} is a solution of $\mbox{MSVI}(F,\mathcal{C}\cap\mathcal{N})$ \eqref{SVI}. When the sum rule \eqref{sum rule} holds, a solution of \eqref{SVI} is also a solution of  \eqref{SVIextensivedefini}. Some sufficient conditions for the sum rule \eqref{sum rule} can be found in \cite[Theorem 3.2]{Rockafellar17} and \cite[Theorem 2.1]{Zhang22}.

\section{Algorithm and Convergence Analysis}

In this section,  we shall propose a Halpern-type relaxed inertial inexact $\mbox{PHA}$ for solving $\mbox{MSVI}(F,\mathcal{C}\cap\mathcal{N})$ \eqref{SVI} and prove that the sequence generated by the algorithm converges strongly in $\mathcal{L}^{2}$ to a solution of $\mbox{MSVI}(F,\mathcal{C}\cap\mathcal{N})$ \eqref{SVI}. We assume

\begin{enumerate}
\item[{(A1)}] The solution set $\mbox{SOL}(F,\mathcal{C}\cap\mathcal{N} )$ of the multi-stage $\mbox{SVI}$ in extensive form \eqref{SVIextensivedefini} is nonempty;
\item[{(A2)}] The mapping $F:\mathcal{L}^{2}\rightarrow \mathcal{L}^{2}$ is monotone and Lipschitz continuous with constant $L_{F}>0$, i.e., for all $x,y\in \mathcal{L}^{2}$,
$\inner{F(x)-F(y)}{x-y}_{\mathcal{L}^{2}}\geq 0$
and
$\norm{F(x)-F(y)}_{\mathcal{L}^{2}}\leq L_{F}\norm{x-y}_{\mathcal{L}^{2}}$.

\end{enumerate}

To clearly illustrate the key idea of our algorithm, let us first make a brief review of the $\mbox{PHA}$ proposed by Rockafellar and Sun in \cite{Rockafellar19}. Let $r>0$ and choose arbitrarily $x^{0}\in \mathcal{N}$ and $y^{0}\in \mathcal{M}$, in the $\mbox{PHA}$, $x^{k}$ and $y^{k}$  are updated by
\begin{equation}\label{PHA}
\left\{
\begin{array}{ll}
-F(\tilde{x}^{k})(\omega)-y^{k}(\omega)+r(x^{k}(\omega)-\tilde{x}^{k}(\omega))\in N_{C(\omega)}(\tilde{x}^{k}(\omega))\ \mbox{a.s.}~\omega\in\Omega, \\
x^{k+1}=P_{\mathcal{N}}(\tilde{x}^{k}), \\
y^{k+1}=y^{k}+rP_{\mathcal{M}}(\tilde{x}^{k}).
\end{array}
\right.
\end{equation}
Let $T$ be the partial inverse of $F+N_{\mathcal{C}}$ with respect to $\mathcal{N}$, i.e.,
\begin{equation}\label{equivalent mapping}
y\in T(x)\Leftrightarrow P_{\mathcal{N}}(y)+P_{\mathcal{M}}(x)\in [F+N_{\mathcal{C}}](P_{\mathcal{N}}(x)+P_{\mathcal{M}}(y)), \ \forall \ (x,y)\in \mathcal{L}^2\times \mathcal{L}^2.
\end{equation}
By Lemma \ref{F and Nk} and condition (A2), we obtain that $F+N_{\mathcal{C}}$ is maximal monotone. Then, by Lemma \ref{TN is maximal monotone},  $T$ is maximal monotone. Define the rescaling mapping $A$ by
\begin{equation}\label{A}
A(z)=P_{\mathcal{N}}(z)+rP_{\mathcal{M}}(z) ,\  \forall~z\in\mathcal{L}^{2}.
\end{equation}
Clearly, $A$ is a symmetric invertible linear operator.  By Lemma~\ref{maxi ATA}, the mapping $ATA$ is also maximal monotone.  It is proved in \cite{Rockafellar19} that the $\mbox{PHA}$~\eqref{PHA} for $\mbox{MSVI}(F,\mathcal{C}\cap\mathcal{N})$ \eqref{SVI} is equivalent to applying the $\mbox{PPA}$ to find zeros of the maximal monotone operator $ATA$.
Exactly, let $u^0=x^0-r^{-1}y^0$ and define the iteration
\begin{equation} \label{PPA}
u^{k+1}=J_{r^{-1}}^{ATA}(u^{k}),
\end{equation}
we have that the sequence $\{(x^{k},y^{k})\}_{k=0}^{\infty}$ generated by $\mbox{PHA}$ \eqref{PHA} can be represented by
$$x^{k}=P_{\mathcal{N}}(A(u^{k})),\quad y^k=-P_{\mathcal{M}}(A(u^{k})), \  \forall~k=0,1,\ldots .$$

Note that, in $\mbox{PHA}$~\eqref{PHA}, it is difficult to compute exactly the intermediate iteration vector  $\tilde{x}^{k}$. From the perspective of numerical calculation, some modified $\mbox{PHAs}$ which allow  for inexact calculation of intermediate iteration vector $\tilde{x}^{k}$ are much easier to implement. On the other hand, by the equivalence between the $\mbox{PHA}$ and the $\mbox{PPA}$, some accelerated approaches for $\mbox{PPA}$ can be used to modify the $\mbox{PHA}$ to make the algorithm converge faster. In order to allow the intermediate iteration vector $\tilde{x}^{k}$ to be solved inexactly and accelerate the algorithm, in what follows, we propose the Halpern-type relaxed inertial inexact $\mbox{PHA}$ for $\mbox{MSVI}(F,\mathcal{C}\cap\mathcal{N})$ \eqref{SVI}.

The update scheme of the Halpern-type relaxed inertial inexact $\mbox{PHA}$ is as follows:

\begin{algorithm}
    \caption{Halpern-type Relaxed Inertial Inexact PHA}
    \label{Algorithm}
    \begin{algorithmic}
        \STATE \textbf{Step 0.} Choose $r>0$, $\{\alpha_{k}\}_{k=0}^{\infty}\subseteq(0,1)$, $\{\beta_{k}\}_{k=0}^{\infty}\subseteq(0,2)$, $\{\theta_{k}\}_{k=0}^{\infty}\subseteq[0,1)$, $\{\varepsilon_{k}\}_{k=0}^{\infty}\subseteq [0,+\infty)$,  $x^{-1},x^{0}\in \mathcal{N}$, $y^{-1},y^{0}\in \mathcal{M}$, and set $k=0$.

        \STATE \textbf{Step 1.} For the given iterate vectors $x^{k-1},x^{k}\in\mathcal{N}$ and $y^{k-1},y^{k}\in\mathcal{M}$, set
  \begin{equation*}
  \hat{x}^{k}=x^{k}+\theta_{k}(x^{k}-x^{k-1})~~\mbox{and}~~\hat{y}^{k}=y^{k}+\theta_{k}(y^{k}-y^{k-1}).
  \end{equation*}

        \STATE \textbf{Step 2.} Compute $\tilde{x}^{k}$ such that
\begin{equation}\label{tilde x}
  \big\|\tilde{x}^{k}-\Pi_{\mathcal{C}}\big\{\hat{x}^{k}
  -r^{-1}[F(\tilde{x}^{k})+\hat{y}^{k}]\big\}\big\|_{\mathcal{L}^{2}}\leq \varepsilon_{k}.
  \end{equation}

        \STATE \textbf{Step 3.} Let $z^{k}=\Pi_{\mathcal{C}}\big\{\hat{x}^{k}
        -r^{-1}[F(\tilde{x}^{k})+\hat{y}^{k}]\big\}$  and update
  \begin{equation}\label{x k+1}
  x^{k+1}=\alpha_{k}x^{0}+(1-\alpha_{k})\big[(1-\beta_{k})\hat{x}^{k}+\beta_{k}P_{\mathcal{N}}(z^{k})\big].
  \end{equation}

        \STATE \textbf{Step 4.} Update
  \begin{equation}\label{omega k+1}
  y^{k+1}=\alpha_{k}y^{0}+(1-\alpha_{k})\Big\{(1-\beta_{k})\hat{y}^{k}
  +\beta_{k}\big[\hat{y}^{k}+rP_{\mathcal{M}}(z^{k})+P_{\mathcal{M}}(e^{k})\big]\Big\},
  \end{equation}
  where $e^{k}=F(\tilde{x}^{k})-F(z^{k})$.
  Let $k\leftarrow k+1$ and return to Step 1.
    \end{algorithmic}
\end{algorithm}

\begin{remark}

The  $\tilde{x}^{k}$ in Step 2 is an inexact solution to the stochastic variational inequality (without nonanticipativity constraint): find $\tilde{x}^*\in \mathcal{C}$ such that
\begin{equation}\label{nonparamVI}
\inner{F(\tilde{x}^* )+\hat y^{k}+r(\tilde{x}^*-\hat x^{k})}{z-\tilde{x}^*}_{\mathcal{L}^2}\ge 0,~\forall~z\in \mathcal{C}.
\end{equation}
In general, variational inequality \eqref{nonparamVI} do not have analytical solutions, some  numerical algorithms for variational inequalities have to be used to solve it, which will inevitably lead to calculation errors. The inequality \eqref{tilde x} is usually used as a stop criterion (for small enough  $\varepsilon_{k}$) for most numerical algorithms for inexact solving variational inequality \eqref{nonparamVI}.

By \cite[Lemma 2.5]{Zhang22}, variational inequality \eqref{nonparamVI} is equivalent to finding $\tilde{x}^*\in \mathcal{C}$ such that
\begin{equation}\label{paramVI}
\inner{F(\tilde{x}^*)(\omega)+\hat{y}^{k}(\omega)+r(\tilde{x}^*(\omega)-\hat{x}^{k}(\omega))}{z-\tilde{x}^*(\omega)}\leq 0,~\forall~z\in C(\omega),\quad a.s.\ \omega\in \Omega.
\end{equation}
Therefore, the $\tilde{x}^{k}$ in Step 2 can be found by solving pointwisely the parameterized variational inequalities \eqref{paramVI} with the value of the residual function
\begin{equation} \label{residual function}
r(\tilde{x}^{k}(\omega))=\big|\tilde{x}^{k}(\omega)-\Pi_{C(\omega)}\big\{\hat{x}^{k}(\omega)
-r^{-1}[F(\tilde{x}^{k})(\omega)+\hat{y}^{k}(\omega)]\big\}\big|
\end{equation}
less than $\varepsilon_{k}$.
\end{remark}

\begin{remark}
\begin{enumerate}[(i)]
  \item If $\theta_{k}=0$ for all $k=0,1,\ldots$, Algorithm~\ref{Algorithm} reduces to the Halpern-type relaxed inexact $\mbox{PHA}$, see \cite[Algorithm 2]{Chen23};
  \item If  $\alpha_{k}=\theta_{k}=0$ for all $k=0,1,\ldots$, Algorithm~\ref{Algorithm} reduces to the relaxed inexact $\mbox{PHA}$, see \cite[Algorithm 1]{Chen23};
  \item If $\varepsilon_{k}=\alpha_{k}=\theta_{k}=0$, $\beta_{k}=1$ for all $k=0,1,\ldots$, Algorithm~\ref{Algorithm} reduces to the original $\mbox{PHA}$ proposed by Rockafellar and Sun in \cite{Rockafellar19}.
\end{enumerate}

The convergence of the algorithms in the above special cases has been proved in the related references. Particularly, the original $\mbox{PHA}$ and the relaxed inexact $\mbox{PHA}$ only exhibit weak convergence in infinite-dimensional spaces. In what follows, we will discuss the more interesting cases that $\varepsilon_{k}\neq 0$, $k=0,1,\ldots$, $\{\alpha_{k}\}_{k=0}^{\infty}\subseteq(0,1)$, $\{\beta_{k}\}_{k=0}^{\infty}\subseteq (0,2)$ and $\{\theta_{k}\}_{k=0}^{\infty}\subseteq [0,1)$. In fact, only when $\varepsilon_{k}\neq 0$, $k=0,1,\ldots$, the proximal subproblem~\eqref{tilde x} is solved inexactly. Moreover, by appropriately selecting parameter $\{\alpha_{k}\}_{k=0}^{\infty}\subseteq(0,1)$, we can extend the weak convergence results of the original $\mbox{PHA}$ and the relaxed inexact $\mbox{PHA}$ to strong convergence in infinite-dimensional spaces. Finally, the over-relaxed parameter $\{\beta_{k}\}_{k=0}^{\infty}\subseteq (1,2)$ and the inertial parameter $\{\theta_{k}\}_{k=0}^{\infty}\subseteq [0,1)$ can accelerate the convergence of Algorithm~\ref{Algorithm}.
\end{remark}

For the parameter sequences $\{\alpha_{k}\}_{k=0}^{\infty}$, $\{\beta_{k}\}_{k=0}^{\infty}$, $\{\theta_{k}\}_{k=0}^{\infty}$  and  $\{\varepsilon_{k}\}_{k=0}^{\infty}$, we assume that

\begin{enumerate}
\item[{(A3)}] $\{\alpha_{k}\}_{k=0}^{\infty}\subseteq(0,1)$ satisfies $\lim\limits_{k\rightarrow\infty}\alpha_{k}=0$ and $\sum\limits_{k=0}^{\infty}\alpha_{k}=\infty$; $\{\beta_{k}\}_{k=0}^{\infty}\subseteq(0,2)$ satisfies $0<\liminf\limits_{k\rightarrow\infty}\beta_{k}\leq\limsup\limits_{k\rightarrow\infty}\beta_{k}<2$; $\{\theta_{k}\}_{k=0}^{\infty}\subseteq[0,1)$ satisfies $\sum\limits_{k=0}^{\infty}\frac{\theta_{k}}{\alpha_{k}}
    \norm{(x^{k}-x^{k-1})-r^{-1}(y^{k}-y^{k-1})}_{\mathcal{L}^{2}}<\infty$ and;  $\{\varepsilon_{k}\}_{k=0}^{\infty}$ satisfies $\sum\limits_{k=0}^{\infty}\varepsilon_{k}<\infty$.
\end{enumerate}

\begin{remark}  \label{Remark 3.3}
Since $\theta_{k}$ is allowed to be determined after $x^{k-1},x^{k},y^{k-1}$ and $y^{k}$ being found, the condition $\sum\limits_{k=0}^{\infty}\frac{\theta_{k}}{\alpha_{k}}
    \norm{(x^{k}-x^{k-1})-r^{-1}(y^{k}-y^{k-1})}_{\mathcal{L}^{2}}<\infty$ is easy to implement in numerical computation. For instance, the parameter sequence $\{\theta_{k}\}_{k=0}^{\infty}\subseteq[0,1)$ satisfying the condition (A3) can be chosen by setting
\begin{equation*}
0\leq\theta_{k}\leq\bar{\theta}_{k}
\end{equation*}
with
\begin{equation*}
  \bar{\theta}_{k}=
  \begin{cases}
  \theta, &  \mbox{if}~x^{k}= x^{k-1}~\mbox{and}~y^{k}= y^{k-1},  \vspace{+0.5em} \\
  \min\big\{\frac{\tau\alpha_{k}}{k^{2}\norm{(x^{k}-x^{k-1})-r^{-1}(y^{k}-y^{k-1})}_{\mathcal{L}^{2}}},\theta\big\}, & \mbox{otherwise},
  \end{cases}
\end{equation*}
for some fixed $\theta\in[0,1)$ and constant $\tau>0$.
\end{remark}

Let $\mathcal{H}$ be a Hilbert space and $\widehat{T}:\mathcal{H}\rightsquigarrow\mathcal{H}$ be a maximal monotone operator. Algorithm \ref{Algorithm} is closely related to the following Halpern-type relaxed inertial inexact $\mbox{PPA}$ for the variational inclusion problem $0\in \widehat{T}(u)$:
\begin{equation}\label{HRIIPPA 4}
\left\{\!\!\!\!\!
\begin{array}{ll}
& \hat{u}^{k}=u^{k}+\theta_{k}(u^{k}-u^{k-1}), \\[+0.5em]
& u^{k+1}=\alpha_{k}u^{0}+(1-\alpha_{k})
\Big[(1-\beta_{k})\hat{u}^{k}+\beta_{k}J_{r^{-1}}^{\widehat{T}}(\hat{u}^{k}-r^{-1}e^{k})\Big],
\end{array}
\right.
\end{equation}
where $r>0$, $u^{-1}, u^{0}\in\mathcal{H}$ are given initial iteration points, $\{\alpha_{k}\}_{k=0}^{\infty}\subseteq(0,1)$, $\{\beta_{k}\}_{k=0}^{\infty}\subseteq(0,2)$, $\{\theta_{k}\}_{k=0}^{\infty}\subseteq[0,1)$ are given parameter sequences, and $\{e^{k}\}_{k=0}^{\infty}\subset\mathcal{H}$ is the error sequence.

We denote by $S$ the set of zeros of the maximal monotone operator $\widehat{T}$, i.e., $$S=\{u\in\mathcal{H}~|~0\in \widehat{T}(u)\},$$
and we always assume that $S$ is nonempty. Then, $S$ is a nonempty closed convex subset so that the metric projection $\Pi_{S}(u)$ is well defined for any $u\in\mathcal{H}$.

\begin{theorem}\label{theorem4}
Suppose that the following conditions hold:
\begin{enumerate}
\item[{(B1)}] $\{\alpha_{k}\}_{k=0}^{\infty}\subseteq(0,1)$ satisfies $\lim\limits_{k\rightarrow\infty}\alpha_{k}=0$ and $\sum\limits_{k=0}^{\infty}\alpha_{k}=\infty$;
\item[{(B2)}] $\{\beta_{k}\}_{k=0}^{\infty}\subseteq(0,2)$ satisfies
$0<\liminf\limits_{k\rightarrow\infty}\beta_{k}\leq\limsup\limits_{k\rightarrow\infty}\beta_{k}<2$;
\item[{(B3)}] $\{\theta_{k}\}_{k=0}^{\infty}\subseteq[0,1)$ satisfies $\sum\limits_{k=0}^{\infty}\frac{\theta_{k}}{\alpha_{k}}\norm{u^{k}-u^{k-1}}_{\mathcal{H}}<\infty$;
\item[{(B4)}] $\{e^{k}\}_{k=0}^{\infty}\subset \mathcal{H}$ satisfies $\sum\limits_{k=0}^{\infty}\norm{e^{k}}_{\mathcal{H}}<\infty$.
\end{enumerate}
Then, the sequence $\{u^{k}\}_{k=0}^{\infty}$ generated by the Halpern-type relaxed inertial inexact $\mbox{PPA}$~\eqref{HRIIPPA 4} converges strongly to the vector $u^{\ast}=\Pi_{S}(u^{0})$.
\end{theorem}
The proof of Theorem \ref{theorem4} will be given in Section 5.

\begin{remark}
Compared with the relaxed inertial PPA studied in \cite{Alvarez04} and \cite{Attouch2020}, in the Halpern-type relaxed inertial inexact $\mbox{PPA}$~\eqref{HRIIPPA 4},   the inexact calculation of the resolvent  and the  Halpern's method for strong  convergence of the algorithm are considered, i.e., the error term $e^{k}\neq 0$ and the
convex combination parameter $\alpha_k\neq 0$ for all $k$. On the other hand,  compared with the modified inertial Mann algorithms in \cite{Tan20},  the Halpern-type relaxed inertial inexact $\mbox{PPA}$~\eqref{HRIIPPA 4}  allows for the inexact calculation of the resolvent and the relaxed parameter $\beta_k\in (1,2)$ for any $k=0,1,...$.
\end{remark}

In what follows we prove that, under proper conditions,   Algorithm~\ref{Algorithm} is equivalent to the Halpern-type relaxed inertial inexact $\mbox{PPA}$ \eqref{HRIIPPA 4} for the maximal monotone operator $\widehat{T}=ATA$ in $\mathcal{L}^2$, where  $T$ and $A$ are given by \eqref{equivalent mapping} and \eqref{A}, respectively.

\begin{theorem}\label{theorem2}
Assume that (A2) is satisfied. Then, Algorithm~\ref{Algorithm} for $\mbox{MSVI}(F,\mathcal{C}\cap\mathcal{N})$ \eqref{SVI} is equivalent to the Halpern-type relaxed inertial inexact $\mbox{PPA}$~\eqref{HRIIPPA 4} for finding zeros of the maximal monotone operator $ATA$ with a special choice of $\{e^k\}_{k=0}^{\infty}$.
\end{theorem}
\begin{proof}
By the definition of $z^{k}$, we have
\begin{equation*}
\big\langle \hat{x}^{k}-r^{-1}[F(\tilde{x}^{k})+\hat{y}^{k}]-z^{k},
z-z^{k}\big\rangle_{\mathcal{L}^{2}}\leq 0,~\forall~z\in\mathcal{C}.
\end{equation*}
It implies that
\begin{equation*}
\hat{x}^{k}-r^{-1}[F(\tilde{x}^{k})+\hat{y}^{k}]-z^{k}\in N_{\mathcal{C}}(z^{k}).
\end{equation*}
Rearranging the above formula, we obtain that
\begin{equation*}
0\in F(\tilde{x}^{k})+\hat{y}^{k}+r(z^{k}-\hat{x}^{k})+N_{\mathcal{C}}(z^{k}).
\end{equation*}
Using the fact $F(\tilde{x}^{k})=F(z^{k})+e^{k}$, we have
\begin{equation*}
0\in F(z^{k})+e^{k}+\hat{y}^{k}+r(z^{k}-\hat{x}^{k})+N_{\mathcal{C}}(z^{k}),
\end{equation*}
or equivalently,
\begin{equation*}
r(\hat{x}^{k}-z^{k})-e^{k}-\hat{y}^{k}\in [F+N_{\mathcal{C}}](z^{k}).
\end{equation*}
Since $z^{k}=P_{\mathcal{N}}(z^{k})+P_{\mathcal{M}}(z^{k})$ and $e^{k}=P_{\mathcal{N}}(e^{k})+P_{\mathcal{M}}(e^{k})$, the above variational inclusion problem can be rewritten as
\begin{equation}\label{variational inclusion problem}
r(\hat{x}^{k}-P_{\mathcal{N}}(z^{k}))-P_{\mathcal{N}}(e^{k})-\hat{y}^{k}-rP_{\mathcal{M}}(z^{k})-P_{\mathcal{M}}(e^{k})\in [F+N_{\mathcal{C}}](z^{k}).
\end{equation}
In addition, by \eqref{x k+1} and \eqref{omega k+1}, we have
\begin{equation} \label{r_x k+1}
P_{\mathcal{N}}(z^{k})=\frac{1}{(1-\alpha_{k})\beta_{k}}x^{k+1}-\frac{1}{\beta_{k}}\frac{\alpha_{k}}{1-\alpha_{k}}x^{0}+(1-\frac{1}{\beta_{k}})\hat{x}^{k}
\end{equation}
and
\begin{equation} \label{r_omega k+1}
\hat{y}^{k}+rP_{\mathcal{M}}(z^{k})+P_{\mathcal{M}}(e^{k})=\frac{1}{(1-\alpha_{k})\beta_{k}}y^{k+1}+(1-\frac{1}{\beta_{k}})\hat{y}^{k}
-\frac{1}{\beta_{k}}\frac{\alpha_{k}}{1-\alpha_{k}}y^{0}.
\end{equation}
Substituting \eqref{r_x k+1} and \eqref{r_omega k+1} into \eqref{variational inclusion problem}, we deduce that 
\begin{equation} \label{Inequality2}
\begin{split}
& r\Big(\frac{1}{\beta_{k}}\hat{x}^{k}-\frac{1}{(1-\alpha_{k})\beta_{k}}x^{k+1}
+\frac{1}{\beta_{k}}\frac{\alpha_{k}}{1-\alpha_{k}}x^{0}\Big)-P_{\mathcal{N}}(e^{k})   \\
-(1 & -\frac{1}{\beta_{k}}) \hat{y}^{k}-\frac{1}{(1-\alpha_{k})\beta_{k}}y^{k+1}+\frac{1}{\beta_{k}}\frac{\alpha_{k}}{1-\alpha_{k}}y^{0}
\in [F+N_{\mathcal{C}}](z^{k}).
\end{split}
\end{equation}
In addition, from \eqref{r_x k+1} and \eqref{r_omega k+1} we have
\begin{equation} \label{Inequality3}
\begin{split}
z^{k}
= & \frac{1}{(1-\alpha_{k})\beta_{k}}x^{k+1}+(1-\frac{1}{\beta_{k}})\hat{x}^{k}-\frac{1}{\beta_{k}}\frac{\alpha_{k}}{1-\alpha_{k}}x^{0}  \\
& + r^{-1}\Big(\frac{1}{(1-\alpha_{k})\beta_{k}}y^{k+1}-\frac{1}{\beta_{k}}\hat{y}^{k}
-\frac{1}{\beta_{k}}\frac{\alpha_{k}}{1-\alpha_{k}}y^{0}-P_{\mathcal{M}}(e^{k})\Big).
\end{split}
\end{equation}
Combining \eqref{Inequality2} with \eqref{Inequality3}, we obtain that
\begin{align*}
\begin{split}
& r\Big(\frac{1}{\beta_{k}}\hat{x}^{k}-\frac{1}{(1-\alpha_{k})\beta_{k}}x^{k+1}
+\frac{1}{\beta_{k}}\frac{\alpha_{k}}{1-\alpha_{k}}x^{0}\Big)-P_{\mathcal{N}}(e^{k})   \\
& -(1-\frac{1}{\beta_{k}})\hat{y}^{k}-\frac{1}{(1-\alpha_{k})\beta_{k}}y^{k+1}+\frac{1}{\beta_{k}}\frac{\alpha_{k}}{1-\alpha_{k}}y^{0} \\
\in & [F+N_{\mathcal{C}}]\Big[\frac{1}{(1-\alpha_{k})\beta_{k}}x^{k+1}+(1-\frac{1}{\beta_{k}})\hat{x}^{k}-\frac{1}{\beta_{k}}\frac{\alpha_{k}}{1-\alpha_{k}}x^{0}  \\
&\qquad\qquad\ + r^{-1}\Big(\frac{1}{(1-\alpha_{k})\beta_{k}}y^{k+1}-\frac{1}{\beta_{k}}\hat{y}^{k}
-\frac{1}{\beta_{k}}\frac{\alpha_{k}}{1-\alpha_{k}}y^{0}-P_{\mathcal{M}}(e^{k})\Big)\Big].
\end{split}
\end{align*}
By \eqref{x k+1}, \eqref{omega k+1}, the definitions of $\hat{x}^{k}$, $\hat{y}^{k}$, and the fact that $x^{-1},x^{0}\in\mathcal{N}$ and $y^{-1},y^{0}\in\mathcal{M}$, we deduce that
$ x^{k},\hat{x}^{k}\in\mathcal{N}$ and $y^{k},\hat{y}^{k}\in\mathcal{M}$  for any $k=0,1,...$. Letting $v^{i}=x^{i}-y^{i}$, $\hat{v}^{i}=\hat{x}^{i}-\hat{y}^{i}$, $i=0,k,k+1$,
we have $x^{i}=P_{\mathcal{N}}(v^{i})$, $\hat{x}^{i}=P_{\mathcal{N}}(\hat{v}^{i})$, $y^{i}=-P_{\mathcal{M}}(v^{i})$ and $\hat{y}^{i}=-P_{\mathcal{M}}(\hat{v}^{i})$, $i=0,k,k+1$. Since the orthogonal projections $P_{\mathcal{N}}$ and $P_{\mathcal{M}}$ are linear operators, we deduce that
\begin{align*}
\begin{split}
& P_{\mathcal{N}}\bigg[r\Big(\frac{1}{\beta_{k}}\hat{v}^{k}-\frac{1}{(1-\alpha_{k})\beta_{k}}v^{k+1}
+\frac{1}{\beta_{k}}\frac{\alpha_{k}}{1-\alpha_{k}}v^{0}-r^{-1}e^{k}\Big)\bigg] \\
& +P_{\mathcal{M}}\Big((1-\frac{1}{\beta_{k}})\hat{v}^{k}
+\frac{1}{(1-\alpha_{k})\beta_{k}}v^{k+1}-\frac{1}{\beta_{k}}\frac{\alpha_{k}}{1-\alpha_{k}}v^{0}\Big) \\
\in & [F+N_{\mathcal{C}}]\bigg\{P_{\mathcal{N}}\Big((1-\frac{1}{\beta_{k}})\hat{v}^{k}
+\frac{1}{(1-\alpha_{k})\beta_{k}}v^{k+1}-\frac{1}{\beta_{k}}\frac{\alpha_{k}}{1-\alpha_{k}}v^{0}\Big) \\
& \qquad\qquad\ +P_{\mathcal{M}}\Big[r^{-1}\Big(\frac{1}{\beta_{k}}\hat{v}^{k}-\frac{1}{(1-\alpha_{k})\beta_{k}}v^{k+1}
+\frac{1}{\beta_{k}}\frac{\alpha_{k}}{1-\alpha_{k}}v^{0}-e^{k}\Big)\Big]\bigg\}.
\end{split}
\end{align*}
By the property of orthogonal projection, we have
\begin{align*}
\begin{split}
& P_{\mathcal{N}}\bigg\{P_{\mathcal{N}}\Big[r\Big(\frac{1}{\beta_{k}}\hat{v}^{k}-\frac{1}{(1-\alpha_{k})\beta_{k}}v^{k+1}
+\frac{1}{\beta_{k}}\frac{\alpha_{k}}{1-\alpha_{k}}v^{0}- r^{-1}e^{k}\Big)\Big] \\
&~~~~~~+P_{\mathcal{M}}\Big[r^{-1}\Big(\frac{1}{\beta_{k}}\hat{v}^{k}
-\frac{1}{(1-\alpha_{k})\beta_{k}}v^{k+1}+\frac{1}{\beta_{k}}\frac{\alpha_{k}}{1-\alpha_{k}}v^{0}-e^{k}\Big)\Big]\bigg\} \\
& +P_{\mathcal{M}}\Big((1-\frac{1}{\beta_{k}})\hat{v}^{k}+\frac{1}{(1-\alpha_{k})\beta_{k}}v^{k+1}-\frac{1}{\beta_{k}}\frac{\alpha_{k}}{1-\alpha_{k}}v^{0}\Big) \\
\in & [F+N_{\mathcal{C}}]\Bigg\{P_{\mathcal{N}}\Big((1-\frac{1}{\beta_{k}})\hat{v}^{k}
+\frac{1}{(1-\alpha_{k})\beta_{k}}v^{k+1}-\frac{1}{\beta_{k}}\frac{\alpha_{k}}{1-\alpha_{k}}v^{0}\Big)  \\
& \qquad\qquad\ +P_{\mathcal{M}}\bigg\{P_{\mathcal{N}}\Big[r\Big(\frac{1}{\beta_{k}}\hat{v}^{k}-\frac{1}{(1-\alpha_{k})\beta_{k}}v^{k+1}
+\frac{1}{\beta_{k}}\frac{\alpha_{k}}{1-\alpha_{k}}v^{0}-r^{-1}e^{k}\Big)\Big] \\
&\qquad\qquad\ ~~~~~~~~~~+P_{\mathcal{M}}\Big[r^{-1}\Big(\frac{1}{\beta_{k}}\hat{v}^{k}-\frac{1}{(1-\alpha_{k})\beta_{k}}v^{k+1}
+\frac{1}{\beta_{k}}\frac{\alpha_{k}}{1-\alpha_{k}}v^{0}-e^{k}\Big)\Big]\bigg\}\Bigg\}.
\end{split}
\end{align*}
Then, by~\eqref{equivalent mapping}, we deduce that
\begin{align*}
\begin{split}
& P_{\mathcal{N}}\Big[r\Big(\frac{1}{\beta_{k}}\hat{v}^{k}-\frac{1}{(1-\alpha_{k})\beta_{k}}v^{k+1}
+\frac{1}{\beta_{k}}\frac{\alpha_{k}}{1-\alpha_{k}}v^{0}- r^{-1}e^{k}\Big)\Big] \\
&+ P_{\mathcal{M}}\Big[r^{-1}\Big(\frac{1}{\beta_{k}}\hat{v}^{k}-\frac{1}{(1-\alpha_{k})\beta_{k}}v^{k+1}
+\frac{1}{\beta_{k}}\frac{\alpha_{k}}{1-\alpha_{k}}v^{0}-e^{k}\Big)\Big] \\
\in & T\Big((1-\frac{1}{\beta_{k}})\hat{v}^{k}+\frac{1}{(1-\alpha_{k})\beta_{k}}v^{k+1}-\frac{1}{\beta_{k}}\frac{\alpha_{k}}{1-\alpha_{k}}v^{0}\Big).
\end{split}
\end{align*}
Consequently,
\begin{align*}
\begin{split}
& P_{\mathcal{N}}\Big[r\Big(\frac{1}{\beta_{k}}\hat{v}^{k}-\frac{1}{(1-\alpha_{k})\beta_{k}}v^{k+1}
+\frac{1}{\beta_{k}}\frac{\alpha_{k}}{1-\alpha_{k}}v^{0}\Big)\Big] \\
& + P_{\mathcal{M}}\Big[r^{-1}\Big(\frac{1}{\beta_{k}}\hat{v}^{k}-\frac{1}{(1-\alpha_{k})\beta_{k}}v^{k+1}
+\frac{1}{\beta_{k}}\frac{\alpha_{k}}{1-\alpha_{k}}v^{0}\Big)\Big] \\
\in & T\Big((1-\frac{1}{\beta_{k}})\hat{v}^{k}+\frac{1}{(1-\alpha_{k})\beta_{k}}v^{k+1}-\frac{1}{\beta_{k}}\frac{\alpha_{k}}{1-\alpha_{k}}v^{0}\Big)
+P_{\mathcal{N}}(e^{k})+r^{-1}P_{\mathcal{M}}(e^{k}).
\end{split}
\end{align*}
By the definition of $A$, we have $A$ is invertible with $A^{-1}(\xi)=P_{\mathcal{N}}(\xi)+r^{-1}P_{\mathcal{M}}(\xi)$ and $A^{-2}(\xi)=A^{-1}(A^{-1}(\xi))= P_{\mathcal{N}}(\xi)+r^{-2}P_{\mathcal{M}}(\xi)$  for any $\xi\in \mathcal{L}^2$. Therefore, the above formula is equivalent to
\begin{equation*}
\begin{split}
& rA^{-2}\Big(\frac{1}{\beta_{k}}\hat{v}^{k}-\frac{1}{(1-\alpha_{k})\beta_{k}}v^{k+1}+\frac{1}{\beta_{k}}\frac{\alpha_{k}}{1-\alpha_{k}}v^{0}\Big)  \\
\in & T\Big((1-\frac{1}{\beta_{k}})\hat{v}^{k}+\frac{1}{(1-\alpha_{k})\beta_{k}}v^{k+1}-\frac{1}{\beta_{k}}\frac{\alpha_{k}}{1-\alpha_{k}}v^{0}\Big)+A^{-1}(e^{k}).
\end{split}
\end{equation*}
Consequently,
\begin{equation*}
\begin{split}
& A^{-1}\Big(\frac{1}{\beta_{k}}\hat{v}^{k}-\frac{1}{(1-\alpha_{k})\beta_{k}}v^{k+1}+\frac{1}{\beta_{k}}\frac{\alpha_{k}}{1-\alpha_{k}}v^{0}\Big) \\
\in & r^{-1}AT\Big((1-\frac{1}{\beta_{k}})\hat{v}^{k}+\frac{1}{(1-\alpha_{k})\beta_{k}}v^{k+1}-\frac{1}{\beta_{k}}\frac{\alpha_{k}}{1-\alpha_{k}}v^{0}\Big)+r^{-1}e^{k}.
\end{split}
\end{equation*}
Define $u^{k}=A^{-1}(v^{k})$, $\hat{u}^{k}=u^{k}+\theta_{k}(u^{k}-u^{k-1})$  for any $k=0,1,...$.  We have $v^{k}=A(u^{k})$ and $\hat{v}^{k}=A(\hat{u}^{k})$  for all $k=0,1,...$.  Then, the above inclusion  can be rewritten as
\begin{equation*}
\begin{split}
& \frac{1}{\beta_{k}}\hat{u}^{k}-\frac{1}{(1-\alpha_{k})\beta_{k}}u^{k+1}+\frac{1}{\beta_{k}}\frac{\alpha_{k}}{1-\alpha_{k}}u^{0} \\
\in & r^{-1}ATA\Big((1-\frac{1}{\beta_{k}})\hat{u}^{k}+\frac{1}{(1-\alpha_{k})\beta_{k}}u^{k+1}-\frac{1}{\beta_{k}}\frac{\alpha_{k}}{1-\alpha_{k}}u^{0}\Big)+r^{-1}e^{k}.
\end{split}
\end{equation*}
It implies that
\begin{equation*}
\hat{u}^{k}\in(I+r^{-1}ATA)\Big((1-\frac{1}{\beta_{k}})\hat{u}^{k}+\frac{1}{(1-\alpha_{k})\beta_{k}}u^{k+1}-\frac{1}{\beta_{k}}\frac{\alpha_{k}}{1-\alpha_{k}}u^{0}\Big)
+r^{-1}e^{k}.
\end{equation*}
From the facts $J_{r^{-1}}^{ATA}=(I+r^{-1}ATA)^{-1}$ and $J_{r^{-1}}^{ATA}$ is a single-valued mapping, we deduce that
\begin{equation*}
(1-\frac{1}{\beta_{k}})\hat{u}^{k}+\frac{1}{(1-\alpha_{k})\beta_{k}}u^{k+1}-\frac{1}{\beta_{k}}\frac{\alpha_{k}}{1-\alpha_{k}}u^{0}
=J_{r^{-1}}^{ATA}(\hat{u}^{k}-r^{-1}e^{k}).
\end{equation*}
Reformulating the above formula, we have
\begin{equation}\label{HRIIPPA 2}
\left\{\!\!\!\!\!\!
\begin{array}{ll}
& \hat{u}^{k}=u^{k}+\theta_{k}(u^{k}-u^{k-1}),\\
& u^{k+1}=\alpha_{k}u^{0}+(1-\alpha_{k})\Big[(1-\beta_{k})\hat{u}^{k}+\beta_{k}J_{r^{-1}}^{ATA}(\hat{u}^{k}-r^{-1}e^{k})\Big],
\end{array}
\right.
\end{equation}
which coincides with the Halpern-type relaxed inertial inexact PPA \eqref{HRIIPPA 4} for the maximal monotone operator $ATA$.

Next, we prove that the converse holds true. Let $u^{0}$, $u^{-1}$ be given vectors in $\mathcal{L}^{2}$,  $u^{k}$, $u^{k-1}$ be the vectors obtained at iterations $k$ and $k-1$, respectively. Define $x^{k}=P_{\mathcal{N}}(A(u^{k}))$ and $y^{k}=-P_{\mathcal{M}}(A(u^{k}))$. Using \eqref{tilde x}, we first obtain the intermediate iteration vectors $\tilde{x}^{k}$ and $z^{k}$. Then, we define $e^{k}=F(\tilde{x}^{k})-F(z^{k})$ and update $u^{k+1}$ by
\begin{equation*}
\left\{\!\!\!\!\!\!
\begin{array}{ll}
& \hat{u}^{k}=u^{k}+\theta_{k}(u^{k}-u^{k-1}),\\
& u^{k+1}=\alpha_{k}u^{0}+(1-\alpha_{k})\Big[(1-\beta_{k})\hat{u}^{k}+\beta_{k}J_{r^{-1}}^{ATA}(\hat{u}^{k}-r^{-1}e^{k})\Big].
\end{array}
\right.
\end{equation*}
By the definition of $J_{r^{-1}}^{ATA}$, we have
\begin{equation}  \label{Inequality19}
\begin{split}
& \frac{1}{\beta_{k}}\hat{u}^{k}-\frac{1}{(1-\alpha_{k})\beta_{k}}u^{k+1}+\frac{1}{\beta_{k}}\frac{\alpha_{k}}{1-\alpha_{k}}u^{0} \\
\in & r^{-1}ATA\Big((1-\frac{1}{\beta_{k}})\hat{u}^{k}+\frac{1}{(1-\alpha_{k})\beta_{k}}u^{k+1}
-\frac{1}{\beta_{k}}\frac{\alpha_{k}}{1-\alpha_{k}}u^{0}\Big)+r^{-1}e^{k}.
\end{split}
\end{equation}
Defining $x^{k+1}=P_{\mathcal{N}}(A(u^{k+1}))$ and $y^{k+1}=-P_{\mathcal{M}}(A(u^{k+1}))$, we have $A(u^{k+1})=x^{k+1}-y^{k+1}$. By the definition of the operator $A$, we deduce that $A^{-1}(u^{k+1})=A^{-1}(A^{-1}(x^{k+1}-y^{k+1}))=x^{k+1}-r^{-2}y^{k+1}$ and $A^{-1}(e^{k})=P_{\mathcal{N}}(e^{k})+r^{-1}P_{\mathcal{M}}(e^{k})$. Similarly, defining  $x^{0}=P_{\mathcal{N}}(A(u^{0}))$, $y^{0}=-P_{\mathcal{M}}(A(u^{0}))$, and, $\hat{x}^{k}=P_{\mathcal{N}}(A(\hat{u}^{k}))$, $\hat{y}^{k}=-P_{\mathcal{M}}(A(\hat{u}^{k}))$, we have that $A(u^{0})=x^{0}-y^{0}$, $A^{-1}(u^{0})=x^{0}-r^{-2}y^{0}$, and,  $A(\hat{u}^{k})=\hat{x}^{k}-\hat{y}^{k}$, $A^{-1}(\hat{u}^{k})=\hat{x}^{k}-r^{-2}\hat{y}^{k}$. Then, the inclusion \eqref{Inequality19} can be rewritten as
\begin{eqnarray*}
&& r\Big[\frac{1}{\beta_{k}}(\hat{x}^{k}-r^{-2}\hat{y}^{k})
-\frac{1}{(1-\alpha_{k})\beta_{k}}(x^{k+1}-r^{-2}y^{k+1})+\frac{1}{\beta_{k}}\frac{\alpha_{k}}{1-\alpha_{k}}(x^{0}-r^{-2}y^{0})\Big] \\
\!\!\!&\in &\!\!\!\! T\Big[(1-\frac{1}{\beta_{k}})(\hat{x}^{k}-\hat{y}^{k})
+\frac{1}{(1-\alpha_{k})\beta_{k}}(x^{k+1}-y^{k+1})-\frac{1}{\beta_{k}}\frac{\alpha_{k}}{1-\alpha_{k}}(x^{0}-y^{0})\Big]\\
&&\,\,
+P_{\mathcal{N}}(e^{k})+r^{-1}P_{\mathcal{M}}(e^{k}).
\end{eqnarray*}
Since $x^{k-1},x^{k},x^{k+1}\in\mathcal{N}$ and $y^{k-1},y^{k},y^{k+1}\in\mathcal{M}$, it follows that
\begin{equation*}
\begin{split}
& P_{\mathcal{N}}\Big[r\Big(\frac{1}{\beta_{k}}\hat{x}^{k}-\frac{1}{(1-\alpha_{k})\beta_{k}}x^{k+1}
+\frac{1}{\beta_{k}}\frac{\alpha_{k}}{1-\alpha_{k}}x^{0}\Big)-e^{k}\Big] \\
& + P_{\mathcal{M}}\Big[r^{-1}\Big(\frac{1}{(1-\alpha_{k})\beta_{k}}y^{k+1}-\frac{1}{\beta_{k}}\hat{y}^{k}
-\frac{1}{\beta_{k}}\frac{\alpha_{k}}{1-\alpha_{k}}y^{0}\Big)-r^{-1}e^{k}\Big]  \\
\in & T\Big((1-\frac{1}{\beta_{k}})(\hat{x}^{k}-\hat{y}^{k})
+\frac{1}{(1-\alpha_{k})\beta_{k}}(x^{k+1}-y^{k+1})-\frac{1}{\beta_{k}}\frac{\alpha_{k}}{1-\alpha_{k}}(x^{0}-y^{0})\Big).
\end{split}
\end{equation*}
Then, by~\eqref{equivalent mapping}, we have
\begin{equation*}
\begin{split}
& P_{\mathcal{N}}\Big[r\Big(\frac{1}{\beta_{k}}\hat{x}^{k}-\frac{1}{(1-\alpha_{k})\beta_{k}}x^{k+1}
+\frac{1}{\beta_{k}}\frac{\alpha_{k}}{1-\alpha_{k}}x^{0}\Big)-e^{k}\Big] \\
& + P_{\mathcal{M}}\Big((1-\frac{1}{\beta_{k}})(\hat{x}^{k}-\hat{y}^{k})
+\frac{1}{(1-\alpha_{k})\beta_{k}}(x^{k+1}-y^{k+1})-\frac{1}{\beta_{k}}\frac{\alpha_{k}}{1-\alpha_{k}}(x^{0}-y^{0})\Big)  \\
\in & [F+N_{\mathcal{C}}]\bigg\{P_{\mathcal{N}}\Big((1-\frac{1}{\beta_{k}})(\hat{x}^{k}-\hat{y}^{k})
+\frac{1}{(1-\alpha_{k})\beta_{k}}(x^{k+1}-y^{k+1})-\frac{1}{\beta_{k}}\frac{\alpha_{k}}{1-\alpha_{k}}(x^{0}-y^{0})\Big)  \\
&\qquad\qquad\ + P_{\mathcal{M}}\Big[r^{-1}\Big(\frac{1}{(1-\alpha_{k})\beta_{k}}y^{k+1}-\frac{1}{\beta_{k}}\hat{y}^{k}
-\frac{1}{\beta_{k}}\frac{\alpha_{k}}{1-\alpha_{k}}y^{0}\Big)-r^{-1}e^{k}\Big]\bigg\},
\end{split}
\end{equation*}
or equivalently,
\begin{equation} \label{Inequality20}
\begin{split}
& r\Big(\frac{1}{\beta_{k}}\hat{x}^{k}-\frac{1}{(1-\alpha_{k})\beta_{k}}x^{k+1}
+\frac{1}{\beta_{k}}\frac{\alpha_{k}}{1-\alpha_{k}}x^{0}\Big)-P_{\mathcal{N}}(e^{k}) \\
& -(1-\frac{1}{\beta_{k}})\hat{y}^{k}-\frac{1}{(1-\alpha_{k})\beta_{k}}y^{k+1}+\frac{1}{\beta_{k}}\frac{\alpha_{k}}{1-\alpha_{k}}y^{0} \\
\in & [F+N_{\mathcal{C}}]\Big[(1-\frac{1}{\beta_{k}})\hat{x}^{k}+\frac{1}{(1-\alpha_{k})\beta_{k}}x^{k+1}-\frac{1}{\beta_{k}}\frac{\alpha_{k}}{1-\alpha_{k}}x^{0} \\
&\qquad\qquad\ +r^{-1}\Big(\frac{1}{(1-\alpha_{k})\beta_{k}}y^{k+1}-\frac{1}{\beta_{k}}\hat{y}^{k}
-\frac{1}{\beta_{k}}\frac{\alpha_{k}}{1-\alpha_{k}}y^{0}-P_{\mathcal{M}}(e^{k})\Big)\Big].
\end{split}
\end{equation}
Defining
\begin{equation} \label{Inequality23}
\begin{split}
\tilde{z}^{k}=& (1-\frac{1}{\beta_{k}})\hat{x}^{k}+\frac{1}{(1-\alpha_{k})\beta_{k}}x^{k+1}-\frac{1}{\beta_{k}}\frac{\alpha_{k}}{1-\alpha_{k}}x^{0} \\
&+   r^{-1}\Big(\frac{1}{(1-\alpha_{k})\beta_{k}}y^{k+1}-\frac{1}{\beta_{k}}\hat{y}^{k}
-\frac{1}{\beta_{k}}\frac{\alpha_{k}}{1-\alpha_{k}}y^{0}-P_{\mathcal{M}}(e^{k})\Big).
\end{split}
\end{equation}
It is obvious that $\tilde{z}^{k}\in \mathcal{C}$. Then, from the facts $x^{k-1},x^{k},x^{k+1}\in\mathcal{N}$ and $y^{k-1},y^{k},y^{k+1}\in\mathcal{M}$, we deduce that
\begin{equation}\label{Inequality21}
x^{k+1}=\alpha_{k}x^{0}+(1-\alpha_{k})\Big[(1-\beta_{k})\hat{x}^{k}+\beta_{k}P_{\mathcal{N}}(\tilde{z}^{k})\Big]
\end{equation}
and
\begin{equation}\label{Inequality22}
y^{k+1}=\alpha_{k}y^{0}+(1-\alpha_{k})\Big\{(1-\beta_{k})\hat{y}^{k}
+\beta_{k}\Big[\hat{y}^{k}+rP_{\mathcal{M}}(\tilde{z}^{k})+P_{\mathcal{M}}(e^{k})\Big]\Big\}.
\end{equation}

Now, we only need to prove $\tilde{z}^{k}=z^{k}$. Substituting \eqref{Inequality23}, \eqref{Inequality21} and \eqref{Inequality22} into \eqref{Inequality20}, we have
\begin{equation*}
r(\hat{x}^{k}-\tilde z^{k})-e^{k}-\hat{y}^{k}\in [F+N_{\mathcal{C}}](\tilde z^{k}).
\end{equation*}
It implies that
\begin{equation}\label{Inequality24}
\inner{F(\tilde z^{k})+r(\tilde z^{k}-\hat{x}^{k})+e^{k}+\hat{y}^{k}}{z-\tilde z^{k}}_{\mathcal{L}^2} \ge 0,\  \forall\ z\in \mathcal{C}.
\end{equation}
In addition, by the definition of $z^{k}$ and the equality $F(\tilde{x}^{k})=F(z^{k})+e^{k}$, we have
\begin{equation} \label{Inequality25}
\inner{F(z^{k})+r(z^{k}-\hat{x}^{k})+e^{k}+\hat{y}^{k}}{z-z^{k}}_{\mathcal{L}^2} \ge 0,\  \forall\ z\in \mathcal{C}.
\end{equation}
By \eqref{Inequality24}, \eqref{Inequality25}  and the monotonicity of $F$,  we deduce that
$$r\|z^k-\tilde z^{k}\|_{\mathcal{L}^2}^2
\le\inner{F(z^{k})-F(\tilde z^{k})+r(z^{k}-\tilde z^{k})}{z^{k}-\tilde z^{k}}_{\mathcal{L}^2}\le 0. $$
Therefore, $z^k=\tilde z^{k}$.

This proves that the Halpern-type Relaxed Inertial Inexact $\mbox{PPA}$~\eqref{HRIIPPA 4} for finding zeros of the maximal monotone operator $ATA$ is equivalent to Algorithm~\ref{Algorithm} if we choose the error sequence $\{e^k\}_{k=0}^{\infty}$ by the above approach. This completes the proof of Theorem \ref{theorem2}.

\end{proof}

Now we prove the convergence of Algorithm~\ref{Algorithm}.
\begin{theorem}\label{theorem5}
Assume that (A1)--(A3) are satisfied. Let $\{(x^{k},y^{k})\}_{k=0}^{\infty}$ be the sequence generated by Algorithm~\ref{Algorithm}. Then $\{(x^{k},y^{k})\}_{k=0}^{\infty}$ converges  strongly to a solution pair $(x^{\ast},y^{\ast})$ of the
multi-stage $\mbox{SVI}$ in extensive form \eqref{SVIextensivedefini}. Especially, $\{x^{k}\}_{k=0}^{\infty}$ converges strongly to a solution $x^{\ast}$ of   $\mbox{MSVI}(F,\mathcal{C}\cap\mathcal{N})$ \eqref{SVI}.
\begin{proof}
We divide the proof into two steps.

Step 1: In this step, we prove that the solution set $\mbox{SOL}(F,\mathcal{C}\cap\mathcal{N})$ of the multi-stage $\mbox{SVI}$ in extensive form  \eqref{SVIextensivedefini} is nonempty if and only if the solution set $S$ of the maximal monotone inclusion problem $0\in ATA(u)$ with $T$ and $A$ defined by \eqref{equivalent mapping} and \eqref{A} is nonempty.

First, we assume that the solution set $\mbox{SOL}(F,\mathcal{C}\cap\mathcal{N} )$ of the multi-stage $\mbox{SVI}$ in extensive form \eqref{SVIextensivedefini} is nonempty, i.e., (A1) holds true. Then, there are $x^*\in \mathcal{N}$ and $y^*\in \mathcal{M}$ such that $-F(x^*)-y^*\in N_{\mathcal{C}}(x^*)$, or equivalently, $P_{\mathcal{M}}(x^*-y^*)\in [F+N_{\mathcal{C}}](P_{\mathcal{N}}(x^*-y^*))$. Letting $u^*=A^{-1}(x^*-y^*)$. Then, we have $A(u^*)=x^*-y^*$ and $P_{\mathcal{M}}(A(u^*))\in [F+N_{\mathcal{C}}](P_{\mathcal{N}}(A(u^*)))$. By \eqref{equivalent mapping}, we have $0\in TA(u^*)$. Since $A$ is a linear invertible mapping, we deduce that $0\in ATA(u^*)$.

Next, assume that the solution set $S$ is nonempty, then there is $u^*$ satisfies $0\in ATA(u^*)$, which implies that $0\in TA(u^*)$. By \eqref{equivalent mapping}, we obtain that $P_{\mathcal{M}}(A(u^*))\in[F+N_{\mathcal{C}}](P_{\mathcal{N}}(A(u^*)))$, which is equivalent to $-F(P_{\mathcal{N}}(A(u^*)))+P_{\mathcal{M}}(A(u^*))\in N_{\mathcal{C}}(P_{\mathcal{N}}(A(u^*)))$. Consequently, the pair $(P_{\mathcal{N}}(A(u^*)),-P_{\mathcal{M}}(A(u^*)))$ is a solution to the multi-stage $\mbox{SVI}$ in extensive form  \eqref{SVIextensivedefini}.

Step 2: In this step we prove that, under conditions (A1)--(A3),
$$(x^{k},y^{k})\rightarrow (P_{\mathcal{N}}(A(u^{\ast})),-P_{\mathcal{M}}(A(u^{\ast})))\in \mbox{SOL}(F,\mathcal{C}\cap \mathcal{N})$$  as $k\rightarrow \infty$, where $u^{\ast}\in S$.

Let $\{(x^{k},y^{k})\}_{k=0}^{\infty}$ be generated by Algorithm~\ref{Algorithm}. Define
$$u^{k}=A^{-1}(x^{k}-y^{k})=x^{k}-r^{-1}y^{k},~ \forall~k=0,1,\ldots.$$
Then by the facts $x^{k}\in\mathcal{N}$ and $y^{k}\in\mathcal{M}$   we have, for all $k$, $(x^{k},y^{k})=(P_{\mathcal{N}}(A(u^{k})),-P_{\mathcal{M}}(A(u^{k})))$. By the proof of Theorem~\ref{theorem2}, we have that $\{u^{k}\}_{k=0}^{\infty}$ can be regarded as a sequence generated by the Halpern-type relaxed inertial inexact $\mbox{PPA}$~\eqref{HRIIPPA 4} for maximal monotone operator $ATA$.

Combining  condition (A2) with the definitions of $e^{k}$ and $\tilde{x}^{k}$ in Algorithm \ref{Algorithm}, we have
\begin{equation*}
\norm{e^{k}}_{\mathcal{L}^{2}}=\norm{F(\tilde{x}^{k})-F(z^{k})}_{\mathcal{L}^{2}}
\leq L_{F}\norm{\tilde{x}^{k}-z^{k}}_{\mathcal{L}^{2}}\leq L_{F}\varepsilon_{k},\ \forall\ k=0,1,\ldots.
\end{equation*}
Then by condition (A3), we deduce that
$
\sum\limits_{k=0}^{\infty}\norm{e^{k}}_{\mathcal{L}^{2}}<\infty.
$
By the fact $\sum\limits_{k=0}^{\infty}\frac{\theta_{k}}{\alpha_{k}}
\norm{(x^{k}-x^{k-1})-r^{-1}(y^{k}-y^{k-1})}_{\mathcal{L}^{2}}<\infty$ in  condition (A3) and the relationship $u^{k}=x^{k}-r^{-1}y^{k}$, we obtain that
\begin{equation*}
\sum\limits_{k=0}^{\infty}\frac{\theta_{k}}{\alpha_{k}}\norm{u^{k}-u^{k-1}}_{\mathcal{L}^{2}}<\infty.
\end{equation*}

By condition (A1) and the conclusion of Step 1, we have  $S\neq \emptyset$. Then, by Theorem \ref{theorem4},  the sequence $\{u^{k}\}_{k=0}^{\infty}$ generated by the Halpern-type relaxed inertial inexact $\mbox{PPA}$ converges strongly to $u^{\ast}=\Pi_{S}(u^{0})$.

Since $P_{\mathcal{N}}$, $P_{\mathcal{M}}$ are bounded linear operators and  $A$ is a symmetric invertible linear mapping, we have $$(x^{k},y^{k})=(P_{\mathcal{N}}(A(u^{k})),-P_{\mathcal{M}}(A(u^{k})))\rightarrow (P_{\mathcal{N}}(A(u^{\ast})),-P_{\mathcal{M}}(A(u^{\ast})))  \text{ as } k\rightarrow \infty.$$

Define $x^{\ast}=P_{\mathcal{N}}(A(u^{\ast}))$, $y^{\ast}=-P_{\mathcal{M}}(A(u^{\ast}))$. Then,  $u^{\ast}=A^{-1}(x^{\ast}-y^{\ast})=x^{\ast}-r^{-1}y^{\ast}$ and
$0\in ATA(x^{\ast}-r^{-1}y^{\ast})$. By the definition of the operator $A$, we have $0\in T(x^{\ast}-y^{\ast})$. Then, by~\eqref{equivalent mapping}, we deduce that $P_{\mathcal{M}}(x^{\ast}-y^{\ast})\in [F+N_{\mathcal{C}}](P_{\mathcal{N}}(x^{\ast}-y^{\ast}))$, or equivalently, $-F(x^{\ast})-y^{\ast}\in N_{\mathcal{C}}(x^{\ast})$, i.e., $(x^{\ast},y^{\ast})\in \mbox{SOL}(F,\mathcal{C}\cap\mathcal{N})$. Especially, $x^{\ast}$ is a solution to $\mbox{MSVI}(F,\mathcal{C}\cap\mathcal{N})$ \eqref{SVI}. This completes the proof of Theorem \ref{theorem5}.

\end{proof}
\end{theorem}

\section{Numerical Example}
In this section, we consider the special case that the sample space is a finite set. On such discrete sample space,  $\mathcal{L}^2$ is isomorphic to a finite-dimensional Hilbert space, see \cite{Rockafellar19,Zhang22} for more details. We test the progressive hedging algorithm ($\mbox{PHA}$ for short, see (2.5) in \cite{Rockafellar19}), the Halpern-type relaxed inexact progressive hedging algorithm ($\mbox{HRIPHA}$ for short, see Algorithm 2 in \cite{Chen23}) and Algorithm~\ref{Algorithm} (Alg.\ref{Algorithm} for short) in such discrete case. All experiments are performed in Matlab2017a on a notebook with Intel(R) Core(TM) i5-5350U CPU (1.80GHz 1.80GHz) and 8.00 GB of RAM.

The following test example is a modification of the one in \cite{Rockafellar19}, see  \cite[Example 4.1]{Zhang22}.

\begin{example}\label{eg:1}
Let $m\in \mathbb{N}$. Define the discrete sample space $\hat{\Omega}:=\{\hat{\omega}_{1},\hat{\omega}_{2},\ldots,\hat{\omega}_{m}\}$, which consists of the finite sample points $\hat{\omega}_{1},\hat{\omega}_{2},\ldots,\hat{\omega}_{m}$. Let $\hat{\mathscr{F}}=2^{\hat{\Omega}}$ be the $\sigma$-algebra generated by all subsets of $\hat{\Omega}$. We randomly generate a discrete probability distribution $\hat{P}(\{\hat{\omega}_{i}\})=p_{i}$, where $p_{i}>0$ for $i=1,2,\ldots,m$ and $\sum\limits_{i=1}^{m}p_{i}=1$. Let $\hat{\xi}_{i}\in R$,  $i=1,2,\ldots,m$, $\hat{\xi}_{i}\neq \hat{\xi}_{j}$ for any $i\neq j$. We define a mapping $\hat{\xi}:\hat{\Omega}\rightarrow R$ by $\hat{\xi}(\hat{\omega}_{i})=\hat{\xi}_{i}$ for $i=1,2,\ldots,m$, and assume that $\hat{P}(\hat{\xi}=\hat{\xi}_{i})=\hat{P}(\{\hat{\omega}_{i}\})=p_{i}$ for $i=1,2,\ldots,m$.
Clearly, $(\hat{\Omega},\hat{\mathscr{F}}, \hat{P})$ forms a complete probability space, and $\hat{\xi}$ is a random vector defined on $(\hat{\Omega},\hat{\mathscr{F}}, \hat{P})$.

In this discrete case, the space of all square-integrable random vectors taking values in $R^{n}$ is defined as
$$\hat{\mathcal{L}}^{2}:=\big\{x: \hat{\Omega}\to R^{n}\ \big|\ \mathds{E}|x|^2= \sum_{i=1}^{m}|x(\omega_i)|^2p_i<+\infty\big\}.$$

Let $n_{0},n_{1}\in \mathbb{N}$ such that $n_{0}+n_{1}=n$. For any $x\in\hat{\mathcal{L}}^2$, we can decompose $x$ as $(x_{0},x_{1})$, where $x_{0}:\hat\Omega\to R^{n_{0}}$ and $x_{1}:\hat\Omega\to R^{n_{1}}$. Define $\hat{\mathscr{F}}_{0}=\{\emptyset, \hat{\Omega}\}$ and $\hat{\mathscr{F}}_{1}=\sigma(\hat{\xi})$. From the definition of $\hat{\xi}$, we see that $\hat{\mathscr{F}}_{1}=2^{\hat{\Omega}}$. The nonanticipativity subspace $\mathcal{N}$ of $\hat{\mathcal{L}}^{2}$ corresponding to $\hat{\mathscr{F}}_{0}$ and $\hat{\mathscr{F}}_{1}$ is represented by
\begin{equation*}
\mathcal{N}=\Big\{x\in \hat{\mathcal{L}}^{2} ~\Big|~x(\hat{\omega}_{i})=(x_{0},x_{1}(\hat{\xi}(\hat{\omega}_{i}))),~i=1,2,\ldots,m\Big\}.
\end{equation*}
This means that the first $n_0$ components of $x \in \mathcal{N}$ are deterministic and equal to some $x_0 \in R^{n_{0}}$, while the last $n_1$ components are functions of $\hat \xi$. By the properties of conditional expectation, the orthogonal complementary subspace $\mathcal{M}$ of $\mathcal{N}$ is given by
\begin{equation*}
\mathcal{M}=\mathcal{N}^{\bot}=\Big\{y\in \hat{\mathcal{L}}^{2}~\Big |~y(\hat{\omega}_{i})=(y_{0}(\hat{\omega}_{i}),0),~i=1,2,\ldots,m,~ \mathds{E}y_{0}=\sum\limits_{i=1}^{m}y_{0}(\hat{\omega}_{i})p_{i}=0\Big\}.
\end{equation*}
In this case, the first $n_0$  components of $y\in  \mathcal{M}$ represent a random vector taking values in $R^{n_{0}}$ with expectation $0$, while the last $n_1$ components are identically zero.

Let $C(\hat{\omega}_{i})\equiv [0,1]^{n}$. The nonempty closed convex set $\mathcal{C}$ can be expressed as
\begin{equation*}
\mathcal{C}=\Big\{x\in \hat{\mathcal{L}}^{2} ~\Big |~x(\hat{\omega}_{i})\in [0,1]^{n},~i=1,2,\ldots,m\Big\}.
\end{equation*}

We define the mapping $\hat{F}:\hat{\mathcal{L}}^{2}\rightarrow \hat{\mathcal{L}}^{2}$ by $\hat{F}(x)=M(\hat{\xi})x+b(\hat{\xi})$ for all $x\in\hat{\mathcal{L}}^{2}$, where $M(\hat{\xi}_{i})=M_{i}+D$. Here, $M_{i}\in R^{n\times n}$ is a randomly generated nonzero positive semi-definite matrix, and $D\in R^{n\times n}$ has diagonal elements ranging from 1 to n, with all other elements of $D$ equaling to zero. Additionally, $b(\hat{\xi}_{i})=b_{i}$, where $b_{i}\in R^{n}$ for $i=1,2,\ldots,m$. It is evident that the mapping $\hat{F}$ is monotone and Lipschitz continuous.

For Alg.\ref{Algorithm}, $\mbox{HRIPHA}$ and $\mbox{PHA}$, we select the initial vectors $x^{-1}=x^{0}\in \mathcal{N}$ and $y^{-1}=y^{0}\in \mathcal{M}$. The stopping criterion is defined as
\begin{equation*}
Err(x^{k}):=\max\limits_{i}\big|x^{k}(\hat{\omega}_{i})-\Pi_{C(\hat{\omega}_{i})}\{x^{k}(\hat{\omega}_{i})
-r^{-1}[\hat{F}(x^{k})(\hat{\omega}_{i})+ y^{k}(\hat{\omega}_{i})]\}\big|<\varepsilon
\end{equation*}
for a given sufficiently small $\varepsilon>0$.

In the implementation of Alg.\ref{Algorithm} and $\mbox{HRIPHA}$, we utilize Malitsky's projected reflected gradient algorithm (see Algorithm 4.1 in \cite{Malitsky2015}) to inexactly solve the corresponding proximal subproblems (e.g., \eqref{tilde x} in Algorithm \ref{Algorithm}) with an accuracy of $\varepsilon_{k}=10^{-4}\times\frac{1}{k^{2}}$ at iteration $k$. For the implementation of $\mbox{PHA}$, the same algorithm is employed to address the proximal subproblems with a fixed accuracy of $\varepsilon_{k}\equiv 10^{-12}$. We define
\begin{equation}  \label{hat-theta}
  \hat{\theta}_{k}=
  \begin{cases}
  \frac{1}{3}, &  \mbox{if}~x^{k}= x^{k-1}~\mbox{and}~y^{k}= y^{k-1}, \\
  \min\big\{\frac{\alpha_{k}}{k^{2}\|(x^{k}-x^{k-1})
  -r^{-1}(y^{k}-y^{k-1})\|_{\mathcal{L}^{2}}}10^{6},~\frac{1}{3}\big\}, & \mbox{otherwise},
  \end{cases}
\end{equation}
where $\{(x^{k}, y^k)\}_{k=0}^{\infty}$ is the sequence generated by Algorithm \ref{Algorithm}.
The details of parameter selections for $\mbox{Alg.1}$, $\mbox{HRIPHA}$ and $\mbox{PHA}$ are listed in Table \ref{Table1}. We show the numerical performance of these algorithms in Table \ref{Table2} and Figures \ref{Figure 1}-\ref{Figure 4}. Each algorithm is run five times. We report the average number of iterations (Avg-iter) and the average running time in seconds (Avg-time(s)). As shown in Table \ref{Table2} and Figures \ref{Figure 1}-\ref{Figure 4}, Alg.\ref{Algorithm} requires less average running time and fewer average number of iterations than the other algorithms for any given accuracy.

\end{example}

\begin{table}[ht!]
\begin{center}
\begin{minipage}{\textwidth}
\caption{Selection of relevant parameters of Alg.\ref{Algorithm},~$\mbox{HRIPHA},~\mbox{PHA}$}\label{Table1}
\vspace{+1em}
\begin{tabular*}{\textwidth}{@{\extracolsep{\fill}}lccccc@{\extracolsep{\fill}}}
\toprule%
\multirow{2}{*}{Algorithms} & \multicolumn{5}{@{}c@{}}{Selection of related parameters} \\  \cmidrule{2-6}%
& $r$ & $\alpha_{k}$ & $\beta_{k}$ & $\theta_{k}$  \\
\midrule
$\mbox{PHA}$  & $\sqrt{n}$ & $\backslash$ & $\backslash$ & $\backslash$      \\\hline
$\mbox{HRIPHA}$  & $\sqrt{n}$ & $\frac{1}{2000(k+1)}$ & $1.5-\frac{1}{(k+1)^2}$ &  $\backslash$        \\\hline
Alg.\ref{Algorithm}  & $\sqrt{n}$ & $\frac{1}{2000(k+1)}$ & $1.5-\frac{1}{(k+1)^2}$ & $\hat{\theta}_{k}$      \\\hline
\end{tabular*}
\end{minipage}
\end{center}
\end{table}

\begin{remark}
The parameters selection in Table \ref{Table1} for Algorithm \ref{Algorithm}  are optimal to some extent, as they not only satisfy the algorithm's convergence conditions but also significantly enhance its performance. The selection methods for $r$, $\alpha_k$, and $\beta_k$ follow those in \cite{Chen23}  where more details are available.  In Remark \ref{Remark 3.3}, we propose a feasible method for selecting the sequence $\{\theta_k\}_{k=0}^{\infty}$, which is closely related to $\alpha_k$ and $k$. Note that $\alpha_k = \frac{1}{2000(k+1)}\to 0$  and $\frac{1}{k^2} \to 0$ as $k \to \infty$, therefore $\theta_k$ will rapidly decrease  to $0$. That will limit the acceleration capability of the inertial term of Algorithm \ref{Algorithm}. Extensive numerical experiments show that $\theta_k = \frac{1}{3}$ optimizes the algorithm's performance. Thus, in practical applications, we choose the $\theta_k$  as detailed in \eqref{hat-theta}, which satisfies Condition (A3) and at the same time maintains the acceleration capability of the inertial term.
\end{remark}

\begin{table}[ht]
\begin{center}
\begin{minipage}{\textwidth}
\caption{Numerical results for Example \ref{eg:1}}\label{Table2}
\vspace{+1em}
\begin{tabular*}{\textwidth}{@{\extracolsep{\fill}}lcccccc@{\extracolsep{\fill}}}
\toprule
$\varepsilon=10^{-3} $
&\multicolumn{2}{@{}c@{}}{$m=50,~n_0=n_1=100 $}
&\multicolumn{2}{@{}c@{}}{$m=75,~n_0=n_1=150 $} \\
\cmidrule{2-3}\cmidrule{4-5}%
 & Avg-time(s) & Avg-iter  &    Avg-time(s) & Avg-iter  \\
\midrule
PHA                 & 100.781  & 114  & 310.047  & 142\\
HRIPHA              & 22.594   & 76   & 69.547   & 95\\
Alg.\ref{Algorithm} & 14.203   & 51   & 40.531   & 63 \\
\toprule
\end{tabular*}

\begin{tabular*}{\textwidth}{@{\extracolsep{\fill}}lcccccc@{\extracolsep{\fill}}}
\toprule
$\varepsilon=10^{-4} $
&\multicolumn{2}{@{}c@{}}{$m=50,~n_0=n_1=100 $}
&\multicolumn{2}{@{}c@{}}{$m=75,~n_0=n_1=150 $} \\
\cmidrule{2-3}\cmidrule{4-5}%
 & Avg-time(s) & Avg-iter  &    Avg-time(s) & Avg-iter  \\
\midrule
PHA                 & 161.219  & 164   & 459.656  & 209\\
HRIPHA              & 35.750   & 111   & 112.984  & 141\\
Alg.\ref{Algorithm} & 21.484   & 73    & 68.828   & 94 \\
\toprule
\end{tabular*}

\begin{tabular*}{\textwidth}{@{\extracolsep{\fill}}lcccccc@{\extracolsep{\fill}}}
\toprule
$\varepsilon=10^{-5}$
& \multicolumn{2}{@{}c@{}}{$m=50,~n_0=n_1=100 $ }
& \multicolumn{2}{@{}c@{}}{$m=75,~n_0=n_1=150$  } \\
\cmidrule{2-3}\cmidrule{4-5}%
& Avg-time(s) & Avg-iter  &    Avg-time(s) & Avg-iter  \\
\midrule
PHA                 & 186.078  & 220 & 610.313 & 275\\
HRIPHA              & 54.578   & 163 & 173.391 & 204\\
Alg.\ref{Algorithm} & 38.891   & 128 & 136.359 & 170 \\
\toprule
\end{tabular*}
\end{minipage}
\end{center}
\end{table}

\begin{figure}[H]
\centering
\begin{minipage}[t]{0.48\textwidth}
\centering
\includegraphics[width=6.5cm]{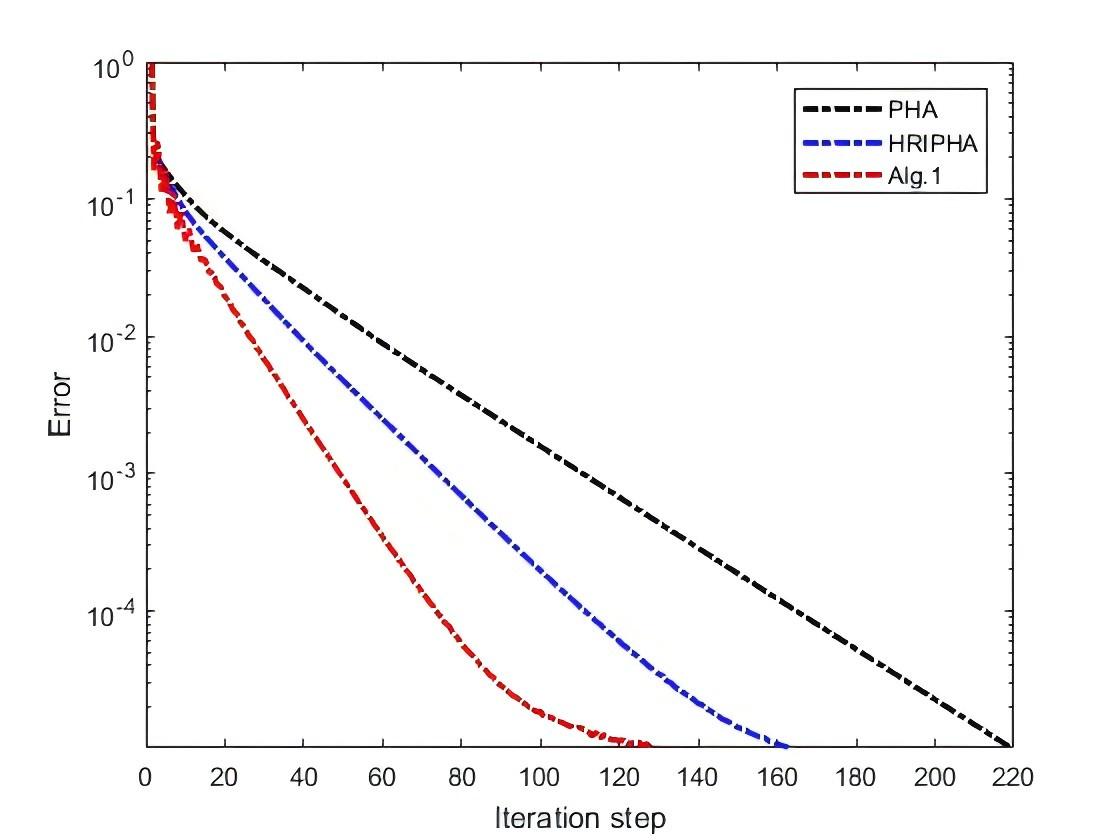}
\vspace{-1.5em}
\caption{$m=50,~n_{0}=n_{1}=100$, $\varepsilon=10^{-5}$}  \label{Figure 1}
\end{minipage}
\begin{minipage}[t]{0.48\textwidth}
\centering
\includegraphics[width=6.5cm]{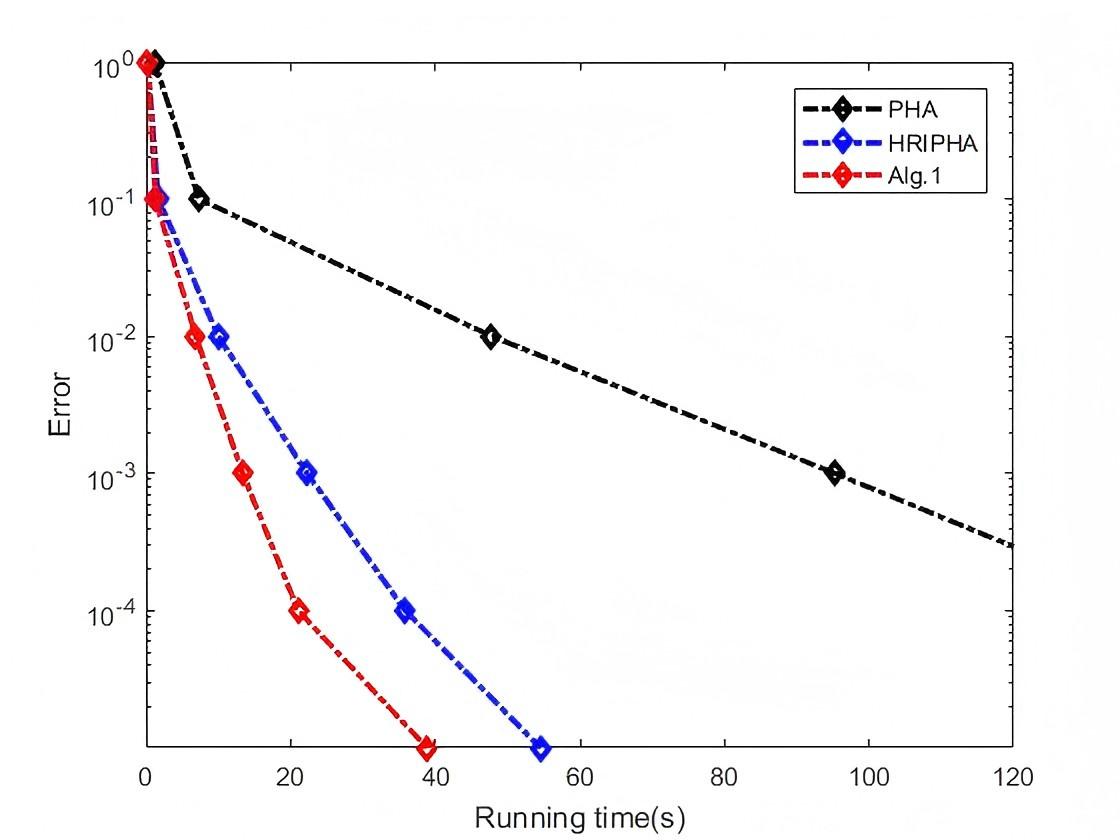}
\vspace{-1.5em}
\caption{$m=50,~n_{0}=n_{1}=100$, $\varepsilon=10^{-5}$}  \label{Figure 2}
\end{minipage}
\end{figure}

\begin{figure}[H]
\centering
\begin{minipage}[t]{0.48\textwidth}
\centering
\includegraphics[width=6.5cm]{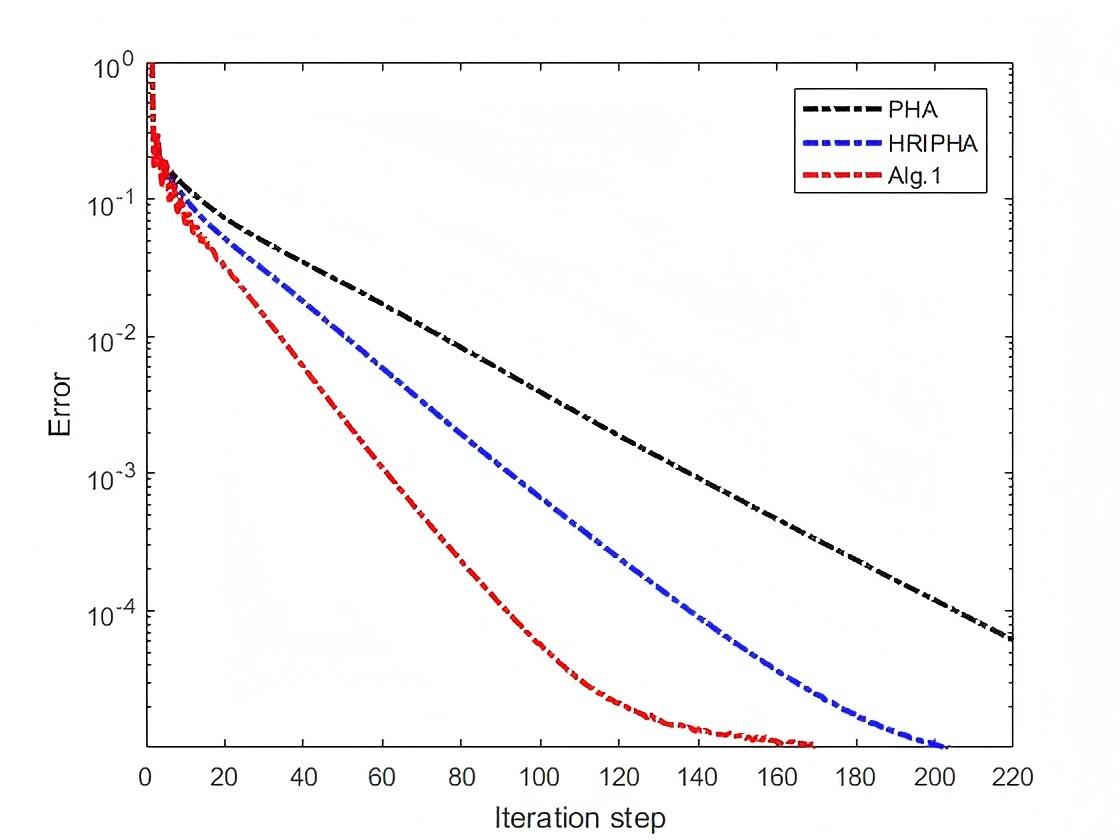}
\vspace{-1.5em}
\caption{$m=75,~n_{0}=n_{1}=150$, $\varepsilon=10^{-5}$}  \label{Figure 3}
\end{minipage}
\begin{minipage}[t]{0.48\textwidth}
\centering
\includegraphics[width=6.5cm]{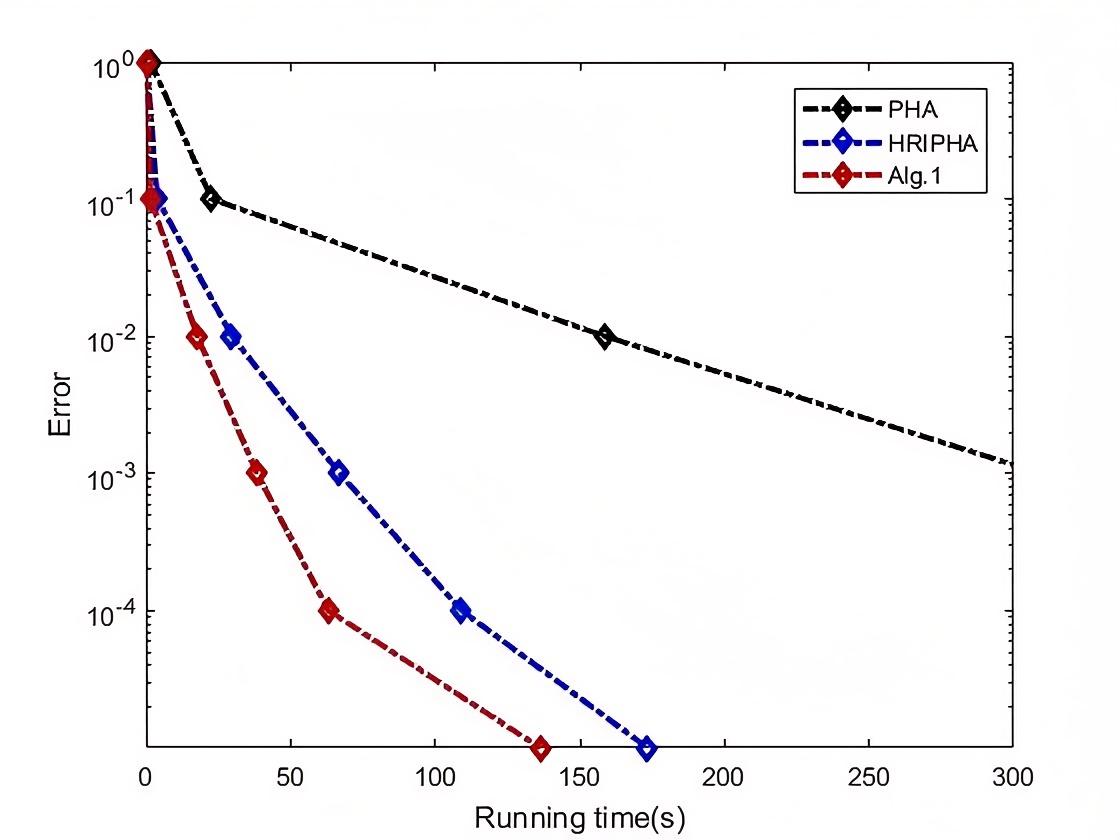}
\vspace{-1.5em}
\caption{$m=75,~n_{0}=n_{1}=150$, $\varepsilon=10^{-5}$}  \label{Figure 4}
\end{minipage}
\end{figure}

The following test example for a discrete stochastic optimal control problem is from \cite[Example 4.2]{Zhang22}.

\begin{example}\label{eg:2}
Let us consider a sequence $\{\xi_{i}\}_{i=1}^{N\ell}$ ($\ell\in \mathbb{N}$) of independent and identically distributed random variables such that for each $i$, $P(\xi_{i}=-1)=P(\xi_{i}=1)=0.5$. We define
$$S_{0}=0,\ S_{i}=\sum_{k=1}^{i}\xi_k,~i=1,2,...,N\ell,$$
and
$$Y_{0}=0,\ Y_{i}^{\ell}=\frac{1}{\sqrt{N\ell}}S_{i\ell},~i=1,2,...,N.$$
Clearly, $(Y_{1}^{\ell},Y_{2}^{\ell},...,Y_{N}^{\ell})$ converges in law to $\big(W(1/N),W(2/N),...,W(1)\big)$ as $\ell\to \infty$ with $W:\Omega\times[0,1]\to R$ being a one-dimensional standard Wiener process defined on some proper filtered probability space.

Define the sample space
\begin{equation}  \label{sample space ex2}
\hat\Omega= \underbrace{\{-1,1\}\times\{-1,1\}\times\cdots\times\{-1,1\}}_{N\ell},
\end{equation}
and let $\hat{\mathscr{F}}$ be the collection of all subsets of $\hat\Omega$  with $\hat P$ being the probability measure induced by the binomial distribution. Thus, $(\hat\Omega,\hat{\mathscr{F}},\hat P)$ forms a discrete probability space. Denote by $\mathcal{L}^2$ the Hilbert space of all the
square-integrable random vectors defined on $(\hat\Omega,\hat{\mathscr{F}},\hat P)$ with values in $R^N$.

Denote
$$\Delta Y_{i}^{\ell}:=(S_{(i+1)\ell}-S_{i\ell})/\sqrt{N\ell}, \ i=0,1,...,N-1,$$
and let $\hat{\mathscr{F}}_0=\{\emptyset,\Omega\}$, $\hat{\mathscr{F}}_{i}=\sigma(\Delta Y_0^{\ell},\Delta Y_1^{\ell},...,\Delta Y_{i-1}^{\ell})$ for $i=1,2,...,N-1$.

 we define
$$
\mathcal{N}\!:=\!\Big\{\hat {\bf u}\!=\!(\hat u_{0},\hat u_{1},...,\hat u_{N-1})\!\in\! \mathcal{L}^2 \ \Big|\ \hat u_{i} \mbox{ is } \hat{\mathscr{F}}_{i} \mbox{-measurable},~ i=0,1,\ldots,N-1\Big\}
$$
and
\begin{equation*}
\mathcal{C}:=\Big\{\hat {\bf u}=(\hat u_{0}, \hat u_{1},...,\hat u_{N-1})\in \mathcal{L}^2\ \Big|\ \hat u_{i}(\omega)\in [0,1],\ a.s.\ \omega\in \hat\Omega,\ i= 0,1,\ldots,N-1 \Big\}.
\end{equation*}

Let $\Delta=1/N$, $\hat{\bf u}=(\hat u_0,\hat u_1,...,\hat u_{N-1})$, $\hat{\bf x}=(\hat x_0,\hat x_1,...,\hat x_N)$,  and consider the following stochastic difference equation
\begin{equation}\label{eq dis_ex_randwalk}
\left\{
\begin{aligned}
&\hat{x}_{i+1}=\hat{x}_{i}+[\hat{x}_{i}-\hat{u}_i]\Delta+ \hat{u}_i\Delta Y_{i}^{\ell}  ,\ i=0,...,N-1,\\[+0.2em]
&\hat{x}_0=1.
\end{aligned}
\right.
\end{equation}

We denote $\hat\Psi_{i}=1+\Delta$ and $\hat\Lambda_{i}=-\Delta+\Delta Y_i^{\ell}$ for $i=0,1,...,N-1$, and define
$$ \hat Z_{i}=\Big[\prod_{j=i+1}^{N-1}\hat{\Psi}_{j}\Big]\hat{\Lambda}_{i},\ i=0,1,..., N-1,\quad
\hat\zeta=\sum_{i=0}^{N-1}\hat Z_{i}.$$
Then, the final value  of the solution to \eqref{eq dis_ex_randwalk} for any $\hat{\bf u}$ is represented as
$\hat x_N=\prod_{i=0}^{N-1}\hat{\Psi}_{i}+\sum_{i=0}^{N-1}\hat Z_{i}\hat{u}_{i}$.
Especially, $\hat \eta=\prod_{i=0}^{N-1}\hat{\Psi}_{i}+ \hat\zeta$ is the final values of the solution to \eqref{eq dis_ex_randwalk} with control   $\hat{\bf u}\equiv(1,1,...,1)$.

Defining the   cost function
\begin{equation}  \label{eq disc cot randwalk}
\hat J^N(\hat {\bf u})
\!=\!\frac{1}{2}\mathbb{E}\big|\hat x_{N}-\hat\eta \big|^2=\frac{1}{2}\mathbb{E}\big|\sum_{i=0}^{N-1}\hat Z_{i}\hat{u}_{i}-\hat\zeta \big|^2, 
\end{equation}
we consider the discrete stochastic optimal control problem: Find $ \hat{\bf u}^*\in \mathcal{C}\cap \mathcal{N}$ such that
\begin{equation}  \label{disc SOCP randomwalk}
\hat J^N(\hat{\bf u}^*)=\min_{\hat{\bf u}\in \mathcal{C}\cap \mathcal{N}}\hat J^N(\hat{\bf u}).
\end{equation}
This problem is a suitable approximation of a continuous stochastic optimal control problem whose control system is a It\^{o} stochastic differential equation and the control region is $[0,1]$, for more details please see \cite[Example 4.2]{Zhang22}. Clearly, $\hat{\bf u}^*\equiv(1,1,...,1)$ is the optimal solution to \eqref{disc SOCP randomwalk}.

Let
\begin{equation*}
M=(\hat Z_0,\hat Z_1,...,\hat Z_{N-1})^{\top}(\hat Z_0,\hat Z_1,...,\hat Z_{N-1}),
\quad
b=\hat\zeta(\hat Z_0,\hat Z_1,...,\hat Z_{N-1})^{\top}.
\end{equation*}
We have
$D \hat J^N(\hat{\bf u})=M\hat{\bf u}-b.$
To solve \eqref{disc SOCP randomwalk}, it suffices to solve the multi-stage stochastic variational inequality
\begin{equation}  \label{SVI_SOCP EX}
D \hat J^N(\hat{\bf u}^*)\in N_{\mathcal{C}\cap\mathcal{N}}(\hat{\bf u}^*).
\end{equation}

Note that the sample space $\hat\Omega$ defined by \eqref{sample space ex2} contains $2^{N\ell}$ sample points.  $\hat\Omega$ will have an extremely large number of sample points when  $\ell$ and/or $N$ sufficiently large. In this case, the Monte Carlo method has to be employed to calculate the projection onto $\mathcal{N}$. Let $\kappa(\in\mathbb{N})$ denote the number of sample size. The parameters are determined in the same manner as those in Table~\ref{Table1}. The numerical results for \eqref{SVI_SOCP EX} are reported in Table \ref{Table3} and Figures \ref{Figure 5}-\ref{Figure 8}. As shown in Table \ref{Table3} and Figures \ref{Figure 5}-\ref{Figure 8}, Algorithm \ref{Algorithm} has a better numerical performance than  $\mbox{PHA}$ and $\mbox{HRIPHA}$.
\end{example}

\begin{table}[ht]
\begin{center}
\begin{minipage}{\textwidth}
\caption{Numerical results for Example \ref{eg:2}}  \label{Table3}
\vspace{+1em}
\begin{tabular*}{\textwidth}{@{\extracolsep{\fill}}lcccccc@{\extracolsep{\fill}}}
\toprule
$\varepsilon=10^{-2} $
&\multicolumn{2}{@{}c@{}}{$N=5,\ell=20,\kappa=500$}
&\multicolumn{2}{@{}c@{}}{$N=5,\ell=40,\kappa=1000$} \\
\cmidrule{2-3}\cmidrule{4-5}%
 & Avg-time(s) & Avg-iter  &    Avg-time(s) & Avg-iter  \\
\midrule
PHA                 & 113.813   & 240  & 473.781   & 452\\
HRIPHA              & 57.016    & 160  & 282.528   & 302\\
Alg.\ref{Algorithm} & 36.688    & 106  & 187.516   & 200\\
\toprule
\end{tabular*}

\begin{tabular*}{\textwidth}{@{\extracolsep{\fill}}lcccccc@{\extracolsep{\fill}}}
\toprule
$\varepsilon=10^{-3} $
&\multicolumn{2}{@{}c@{}}{$N=5,\ell=20,\kappa=500$}
&\multicolumn{2}{@{}c@{}}{$N=5,\ell=40,\kappa=1000$} \\
\cmidrule{2-3}\cmidrule{4-5}%
 & Avg-time(s) & Avg-iter  &    Avg-time(s) & Avg-iter  \\
\midrule
PHA                  & 498.328  & 1010 & 2269.896 & 2228\\
HRIPHA               & 285.172  & 674  & 1634.788 & 1485\\
Alg.\ref{Algorithm}  & 180.734  & 461  & 1227.888 & 1248\\
\toprule
\end{tabular*}
\end{minipage}
\end{center}
\end{table}

\begin{figure}[ht!]
\centering
\begin{minipage}[t]{0.48\textwidth}
\centering
\includegraphics[width=6.5cm]{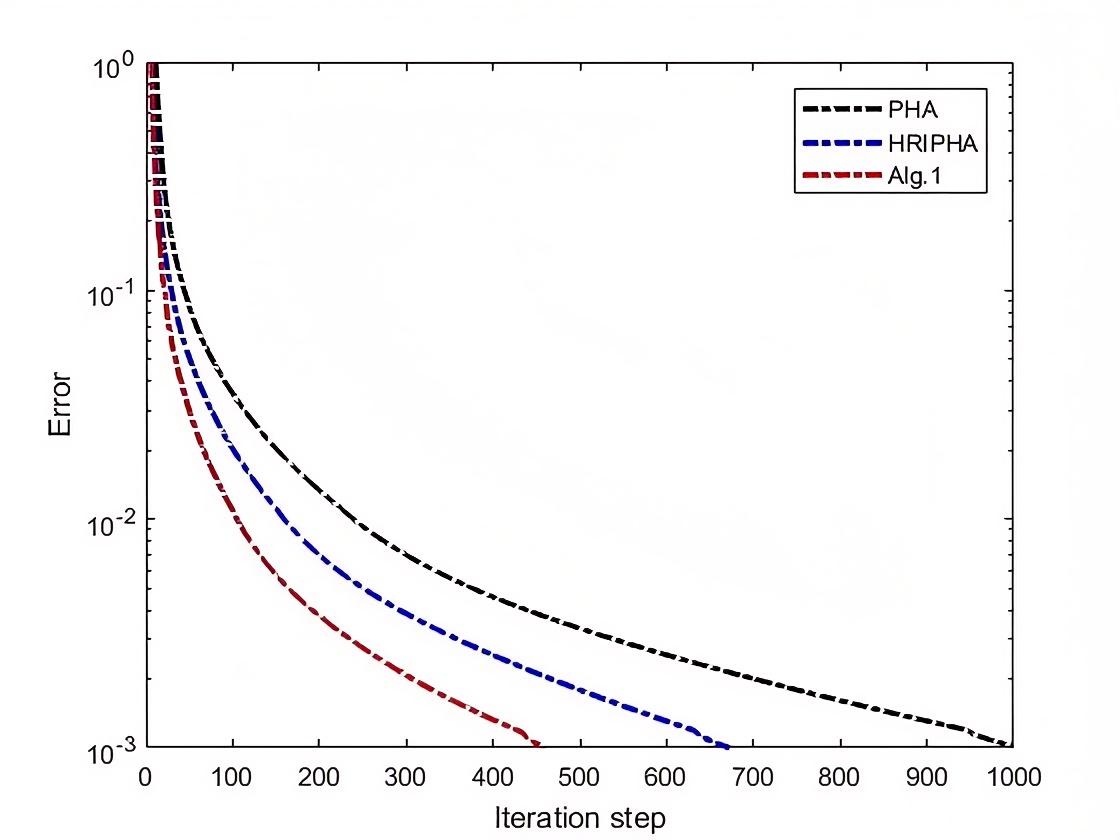}
\vspace{-1.5em}
\caption{$N=5,\ell=20,\kappa=500$,~$\varepsilon=10^{-3}$}  \label{Figure 5}
\end{minipage}
\begin{minipage}[t]{0.48\textwidth}
\centering
\includegraphics[width=6.5cm]{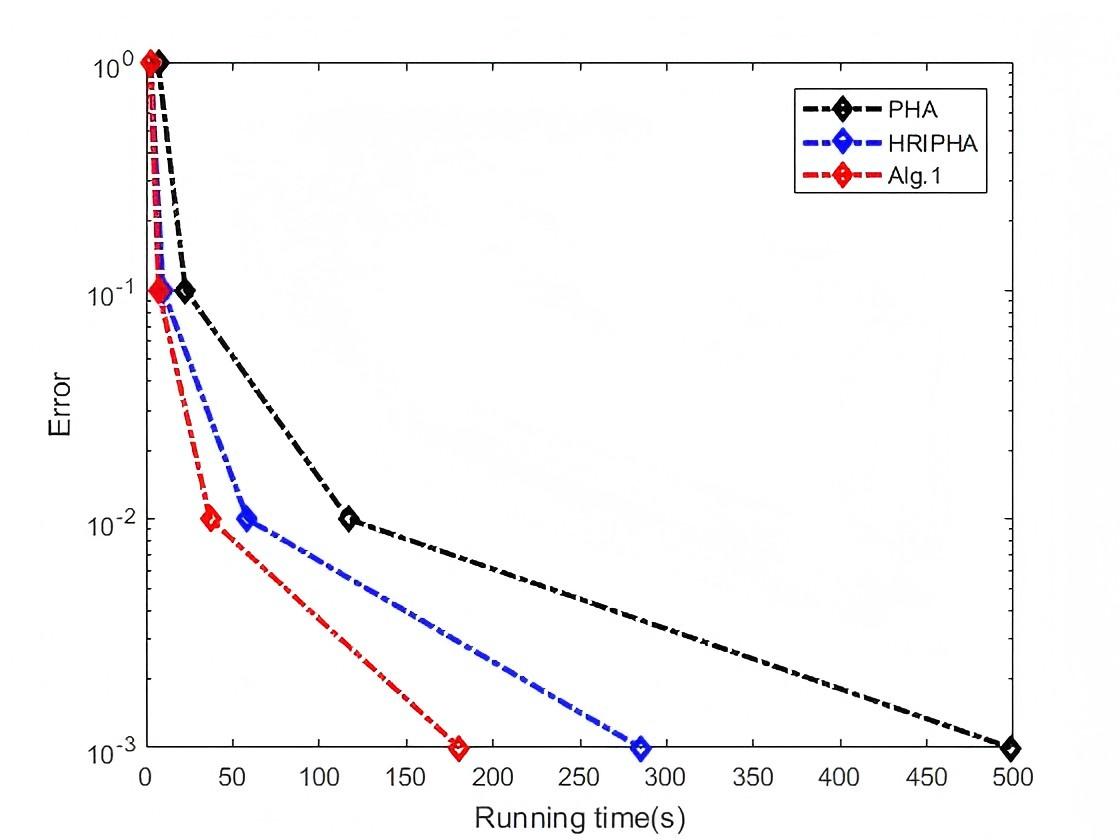}
\vspace{-1.5em}
\caption{$N=5,\ell=20,\kappa=500$,~$\varepsilon=10^{-3}$}  \label{Figure 6}
\end{minipage}
\end{figure}

\vspace{+0.5em}
\begin{figure}[ht!]
\centering
\begin{minipage}[t]{0.48\textwidth}
\centering
\includegraphics[width=6.5cm]{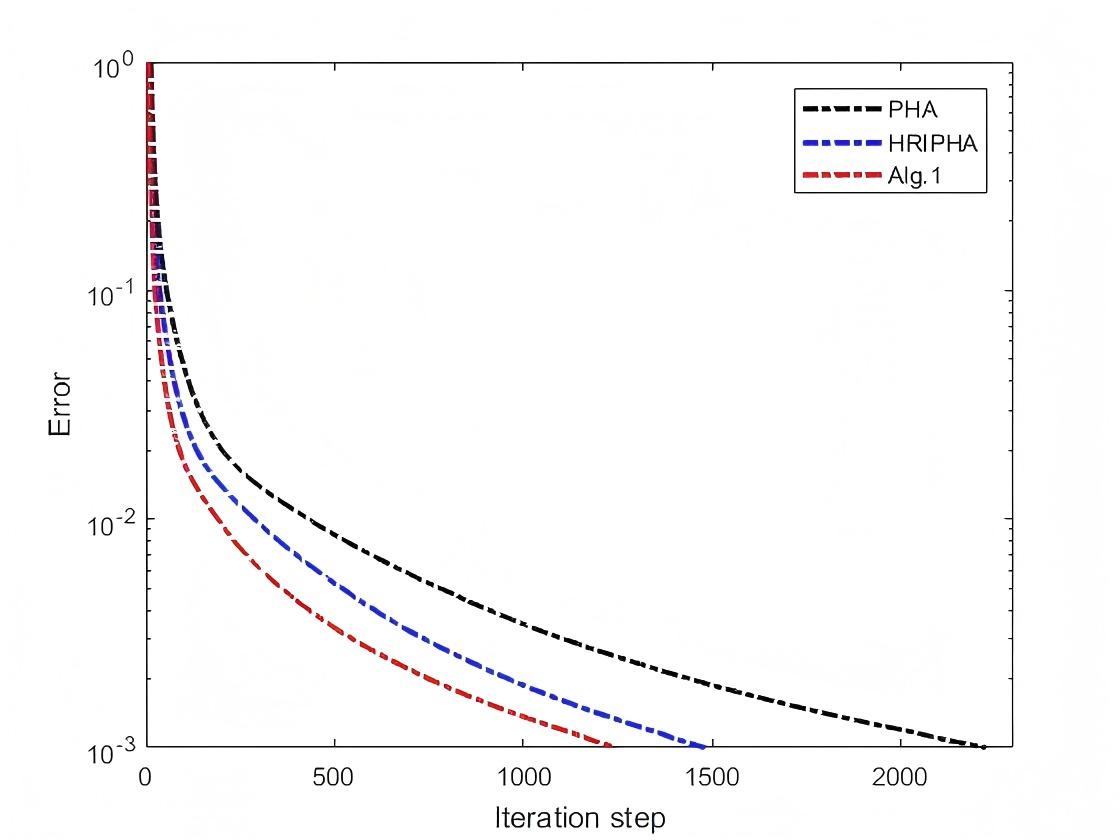}
\vspace{-1.5em}
\caption{$N=5,\ell=40,\kappa=1000$,~$\varepsilon=10^{-3}$}  \label{Figure 7}
\end{minipage}
\begin{minipage}[t]{0.48\textwidth}

\centering
\includegraphics[width=6.5cm]{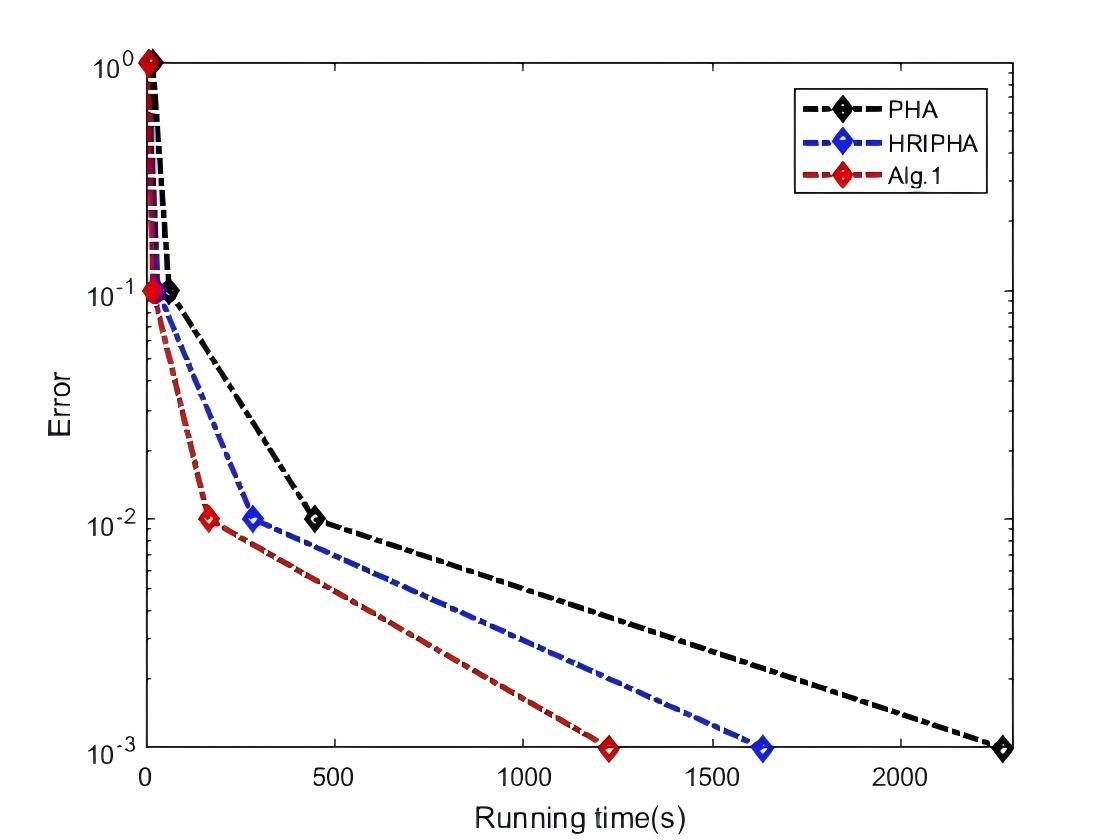}
\vspace{-1.5em}
\caption{$N=5,\ell=40,\kappa=1000$,~$\varepsilon=10^{-3}$}  \label{Figure 8}
\end{minipage}
\end{figure}

\section{The Proof of Theorem~\ref{theorem4}}
This section is devoted to proving Theorem~\ref{theorem4}.

{\bf Proof of Theorem \ref{theorem4}.}
Let us consider the following exact version of the Halpern-type relaxed inertial inexact $\mbox{PPA}$~\eqref{HRIIPPA 4}:
\begin{equation}\label{HRIIPPA 3}
\left\{\!\!\!\!\!\!\!
\begin{array}{ll}
& \hat{\vartheta}^{k}=\vartheta^{k}+\lambda_{k}(\vartheta^{k}-\vartheta^{k-1}), \\[+0.5em]
& \vartheta^{k+1}=\alpha_{k}\vartheta^{0}
+(1-\alpha_{k})\Big[(1-\beta_{k})\hat{\vartheta}^{k}+\beta_{k}J_{r^{-1}}^{\widehat{T}}(\hat{\vartheta}^{k})\Big],
\end{array}
\right.
\end{equation}
where $r>0$, $\vartheta^{-1}=u^{-1}$, $\vartheta^{0}=u^{0}$, the parameter sequences $\{\alpha_{k}\}_{k=0}^{\infty}\subseteq(0,1)$ and $\{\beta_{k}\}_{k=0}^{\infty}\subseteq(0,2)$  satisfy conditions (B1)-(B2) and, the parameter sequence  $\{\lambda_{k}\}_{k=0}^{\infty}\subseteq[0,1)$ satisfies
\begin{equation}\label{B3'}
\sum\limits_{k=0}^{\infty}\frac{\lambda_{k}}{\alpha_{k}}
      \norm{\vartheta^{k}-\vartheta^{k-1}}_{\mathcal{H}}<\infty.
\end{equation}

We first prove that  the sequence $\{\vartheta^{k}\}_{k=0}^{\infty}$ generated by the exact version of the Halpern-type relaxed inertial inexact $\mbox{PPA}$~\eqref{HRIIPPA 4} converges strongly to the vector $\vartheta^{\ast}=\Pi_{S}(\vartheta^{0})$. We divide our proof into four steps.

Step 1. In this step, we prove that the sequences $\{\vartheta^{k}\}_{k=0}^{\infty}$ and $\{\hat{\vartheta}^{k}\}_{k=0}^{\infty}$ are bounded.

Letting $\nu^{k}=(1-\beta_{k})\hat{\vartheta}^{k}+\beta_{k}J_{r^{-1}}^{\widehat{T}}(\hat{\vartheta}^{k})$. By Lemma~\ref{Norm proposition} (ii), we have
\begin{eqnarray} \label{Inequality9}
\norm{\nu^{k}-\vartheta^{\ast}}_{\mathcal{H}}^{2}
\!\!\!\!\! &=&\!\!\!\!\!  \norm{(1-\beta_{k})\hat{\vartheta}^{k}
+\beta_{k}J_{r^{-1}}^{\widehat{T}}(\hat{\vartheta}^{k})-\vartheta^{\ast}}_{\mathcal{H}}^{2} \nonumber\\
\!\!\!\!\! & = &\!\!\!\!\! \norm{(1-\beta_{k})(\hat{\vartheta}^{k}-\vartheta^{\ast})
+\beta_{k}(J_{r^{-1}}^{\widehat{T}}(\hat{\vartheta}^{k})-\vartheta^{\ast})}_{\mathcal{H}}^{2} \nonumber\\
\!\!\!\!\! & =&\!\!\!\!\! (1-\beta_{k})\norm{\hat{\vartheta}^{k}-\vartheta^{\ast}}_{\mathcal{H}}^{2}
+\beta_{k}\norm{J_{r^{-1}}^{\widehat{T}}(\hat{\vartheta}^{k})-\vartheta^{\ast}}_{\mathcal{H}}^{2}
-\beta_{k}(1-\beta_{k})\norm{J_{r^{-1}}^{\widehat{T}}(\hat{\vartheta}^{k})-\hat{\vartheta}^{k}}_{\mathcal{H}}^{2}.\qquad
\end{eqnarray}
Since $\vartheta^{\ast}\in S$, we have $\vartheta^{\ast}=J_{r^{-1}}^{\widehat{T}}(\vartheta^{\ast})$. Then, by Proposition~\ref{Convergence proposition} (v), we deduce that
\begin{equation} \label{Inequality10}
\begin{split}
\norm{J_{r^{-1}}^{\widehat{T}}(\hat{\vartheta}^{k})-\vartheta^{\ast}}_{\mathcal{H}}^{2}
& = \norm{J_{r^{-1}}^{\widehat{T}}(\hat{\vartheta}^{k})-J_{r^{-1}}^{\widehat{T}}(\vartheta^{\ast})}_{\mathcal{H}}^{2} \\
& \leq \norm{\hat{\vartheta}^{k}-\vartheta^{\ast}}_{\mathcal{H}}^{2}
-\norm{(I-J_{r^{-1}}^{\widehat{T}})(\hat{\vartheta}^{k})-(I-J_{r^{-1}}^{\widehat{T}})(\vartheta^{\ast})}_{\mathcal{H}}^{2} \\
& = \norm{\hat{\vartheta}^{k}-\vartheta^{\ast}}_{\mathcal{H}}^{2}
-\norm{\hat{\vartheta}^{k}-J_{r^{-1}}^{\widehat{T}}(\hat{\vartheta}^{k})}_{\mathcal{H}}^{2}.
\end{split}
\end{equation}
Combining \eqref{Inequality9} with \eqref{Inequality10}, we have
\begin{equation} \label{Inequality11}
\begin{split}
\norm{\nu^{k}-\vartheta^{\ast}}_{\mathcal{H}}^{2}
\leq & (1-\beta_{k})\norm{\hat{\vartheta}^{k}-\vartheta^{\ast}}_{\mathcal{H}}^{2}
+\beta_{k}\norm{\hat{\vartheta}^{k}-\vartheta^{\ast}}_{\mathcal{H}}^{2} \\
& -\beta_{k}\norm{\hat{\vartheta}^{k}-J_{r^{-1}}^{\widehat{T}}(\hat{\vartheta}^{k})}_{\mathcal{H}}^{2}
-\beta_{k}(1-\beta_{k})\norm{J_{r^{-1}}^{\widehat{T}}(\hat{\vartheta}^{k})-\hat{\vartheta}^{k}}_{\mathcal{H}}^{2} \\
= & \norm{\hat{\vartheta}^{k}-\vartheta^{\ast}}_{\mathcal{H}}^{2}
-\beta_{k}(2-\beta_{k})\norm{J_{r^{-1}}^{\widehat{T}}(\hat{\vartheta}^{k})-\hat{\vartheta}^{k}}_{\mathcal{H}}^{2}.
\end{split}
\end{equation}
Since $\{\beta_{k}\}_{k=0}^{\infty}\subseteq (0,2)$, we have
$$\norm{\nu^{k}-\vartheta^{\ast}}_{\mathcal{H}}\leq \norm{\hat{\vartheta}^{k}-\vartheta^{\ast}}_{\mathcal{H}}.$$
By the definition of $\hat{\vartheta}^{k}$, we deduce that
\begin{eqnarray} \label{Inequality7}
\norm{\nu^{k}-\vartheta^{\ast}}_{\mathcal{H}}
\!\!\!\!\! & \leq&\!\!\!\!\!  \norm{\vartheta^{k}+\lambda_{k}(\vartheta^{k}-\vartheta^{k-1})-\vartheta^{\ast}}_{\mathcal{H}} \nonumber\\
\!\!\!\!\! &\leq&\!\!\!\!\!   \norm{\vartheta^{k}-\vartheta^{\ast}}_{\mathcal{H}}
+\lambda_{k}\norm{\vartheta^{k}-\vartheta^{k-1}}_{\mathcal{H}}.
\end{eqnarray}
Then, it follows that
\begin{equation} \label{Inequality8}
\begin{split}
\norm{\vartheta^{k+1}-\vartheta^{\ast}}_{\mathcal{H}} & = \norm{\alpha_{k}\vartheta^{0}+(1-\alpha_{k})\nu^{k}-\vartheta^{\ast}}_{\mathcal{H}} \\
& = \norm{\alpha_{k}(\vartheta^{0}-\vartheta^{\ast})
+(1-\alpha_{k})(\nu^{k}-\vartheta^{\ast})}_{\mathcal{H}} \\
& \leq \alpha_{k}\norm{\vartheta^{0}-\vartheta^{\ast}}_{\mathcal{H}}
+(1-\alpha_{k})\norm{\nu^{k}-\vartheta^{\ast}}_{\mathcal{H}} \\
& \leq \alpha_{k}\norm{\vartheta^{0}-\vartheta^{\ast}}_{\mathcal{H}}
+(1-\alpha_{k})\norm{\vartheta^{k}-\vartheta^{\ast}}_{\mathcal{H}}
+\lambda_{k}(1-\alpha_{k})\norm{\vartheta^{k}-\vartheta^{k-1}}_{\mathcal{H}} \\
& \leq \alpha_{k}\norm{\vartheta^{0}-\vartheta^{\ast}}_{\mathcal{H}}
+(1-\alpha_{k})\norm{\vartheta^{k}-\vartheta^{\ast}}_{\mathcal{H}}
+\lambda_{k}\norm{\vartheta^{k}-\vartheta^{k-1}}_{\mathcal{H}} \\
& =(1-\alpha_{k})\norm{\vartheta^{k}-\vartheta^{\ast}}_{\mathcal{H}}
+\alpha_{k}\big(\norm{\vartheta^{0}-\vartheta^{\ast}}_{\mathcal{H}}
+\frac{\lambda_{k}}{\alpha_{k}}\norm{\vartheta^{k}-\vartheta^{k-1}}_{\mathcal{H}}\big).
\end{split}
\end{equation}
By \eqref{B3'}, we obtain that $\frac{\lambda_{k}}{\alpha_{k}}\norm{\vartheta^{k}-\vartheta^{k-1}}_{\mathcal{H}}$ converges to 0 as $k\rightarrow\infty$, hence, $\sup\limits_{k\geq0}\frac{\lambda_{k}}{\alpha_{k}}\norm{\vartheta^{k}-\vartheta^{k-1}}_{\mathcal{H}}$ is bounded. Let $$M_{1}=2\mbox{max}\Big\{\norm{\vartheta^{0}-\vartheta^{\ast}}_{\mathcal{H}},\
\sup\limits_{k\geq0}\frac{\lambda_{k}}{\alpha_{k}}\norm{\vartheta^{k}-\vartheta^{k-1}}_{\mathcal{H}}\Big\}.$$ Then, by \eqref{Inequality8}, we have
\begin{equation*}
\norm{\vartheta^{k+1}-\vartheta^{\ast}}_{\mathcal{H}}\leq (1-\alpha_{k})\norm{\vartheta^{k}-\vartheta^{\ast}}_{\mathcal{H}}
+\alpha_{k}M_{1}\leq \mbox{max}\big\{\norm{\vartheta^{k}-\vartheta^{\ast}}_{\mathcal{H}},M_{1}\big\}.
\end{equation*}
It follows that
$\norm{\vartheta^{k+1}-\vartheta^{\ast}}_{\mathcal{H}}\leq \mbox{max}\big\{\norm{\vartheta^{0}-\vartheta^{\ast}}_{\mathcal{H}},M_{1}\big\}$
and the sequence $\{\vartheta^{k}\}_{k=0}^{\infty}$ is bounded. Moreover, by the definition of $\hat{\vartheta}^{k}$, we have $\{\hat{\vartheta}^{k}\}_{k=0}^{\infty}$ is also bounded.

Step 2. In this step, we prove that
\begin{equation} \label{Inequality4}
\norm{\vartheta^{k+1}-\vartheta^{\ast}}_{\mathcal{H}}^{2}\leq (1-\alpha_{k})\norm{\vartheta^{k}-\vartheta^{\ast}}_{\mathcal{H}}^{2}+\alpha_{k}\gamma_{k},
\end{equation}
where
\begin{equation*}
\begin{split}
\gamma_{k}
= & \alpha_{k}(1-\alpha_{k})\frac{\lambda_{k}^{2}}{\alpha_{k}^{2}}\norm{\vartheta^{k}-\vartheta^{k-1}}_{\mathcal{H}}^{2}
+ 2(1-\alpha_{k})\frac{\lambda_{k}}{\alpha_{k}}
\inner{\vartheta^{k}-\vartheta^{k-1}}{\vartheta^{k}-\vartheta^{\ast}}_{\mathcal{H}} \\
& -\frac{1-\alpha_{k}}{\alpha_{k}}\beta_{k}(2-\beta_{k})
\norm{J_{r^{-1}}^{\widehat{T}}(\hat{\vartheta}^{k})-\hat{\vartheta}^{k}}_{\mathcal{H}}^{2}
+2\inner{\vartheta^{0}-\vartheta^{\ast}}{\vartheta^{k+1}-\vartheta^{\ast}}_{\mathcal{H}}.
\end{split}
\end{equation*}

By the definition of $\vartheta^{k+1}$ and Lemma~\ref{Norm proposition}~(i), we have
\begin{equation*}
\begin{split}
& \norm{\vartheta^{k+1}-\vartheta^{\ast}}_{\mathcal{H}}^{2} \\
= & \norm{\alpha_{k}\vartheta^{0}+(1-\alpha_{k})\nu^{k}-\vartheta^{\ast}}_{\mathcal{H}}^{2}
\\
=&\norm{\alpha_{k}(\vartheta^{0}-\vartheta^{\ast})+(1-\alpha_{k})(\nu^{k}-\vartheta^{\ast})}_{\mathcal{H}}^{2} \\
 \leq &(1-\alpha_{k})^{2}\norm{\nu^{k}-\vartheta^{\ast}}_{\mathcal{H}}^{2}
+2\alpha_{k}\inner{\vartheta^{0}-\vartheta^{\ast}}{\vartheta^{k+1}-\vartheta^{\ast}}_{\mathcal{H}} \\
\leq&  (1-\alpha_{k})\norm{\nu^{k}-\vartheta^{\ast}}_{\mathcal{H}}^{2}
+2\alpha_{k}\inner{\vartheta^{0}-\vartheta^{\ast}}{\vartheta^{k+1}-\vartheta^{\ast}}_{\mathcal{H}}.
\end{split}
\end{equation*}
Then, by \eqref{Inequality11} and the definition of $\hat{\vartheta}^{k}$, we deduce that
\begin{eqnarray*}
\!\!\!&&\!\!\! \norm{\vartheta^{k+1}-\vartheta^{\ast}}_{\mathcal{H}}^{2} \\
\!\!\!& \leq &\!\!\! (1-\alpha_{k})\norm{\hat{\vartheta}^{k}-\vartheta^{\ast}}_{\mathcal{H}}^{2}
-(1-\alpha_{k})\beta_{k}(2-\beta_{k})\norm{J_{r^{-1}}^{\widehat{T}}(\hat{\vartheta}^{k})-\hat{\vartheta}^{k}}_{\mathcal{H}}^{2}\\
\!\!\!&&\!\!\!+2\alpha_{k}\inner{\vartheta^{0}-\vartheta^{\ast}}{\vartheta^{k+1}-\vartheta^{\ast}}_{\mathcal{H}} \\
\!\!\!&=&\!\!\! (1-\alpha_{k})\norm{\vartheta^{k}-\vartheta^{\ast}}_{\mathcal{H}}^{2}
+(1-\alpha_{k})\lambda_{k}^{2}\norm{\vartheta^{k}-\vartheta^{k-1}}_{\mathcal{H}}^{2}\\
\!\!\!&&\!\!\!+2(1-\alpha_{k})\lambda_{k}\inner{\vartheta^{k}-\vartheta^{k-1}}{\vartheta^{k}-\vartheta^{\ast}}_{\mathcal{H}} \\
\!\!\!&&\!\!\!  -(1-\alpha_{k})\beta_{k}(2-\beta_{k})\norm{J_{r^{-1}}^{\widehat{T}}(\hat{\vartheta}^{k})-\hat{\vartheta}^{k}}_{\mathcal{H}}^{2}
+2\alpha_{k}\inner{\vartheta^{0}-\vartheta^{\ast}}{\vartheta^{k+1}-\vartheta^{\ast}}_{\mathcal{H}} \\
\!\!\!&=&\!\!\! (1-\alpha_{k})\norm{\vartheta^{k}-\vartheta^{\ast}}_{\mathcal{H}}^{2}
+\alpha_{k}\Big[\alpha_{k}(1-\alpha_{k})\frac{\lambda_{k}^{2}}{\alpha_{k}^{2}}\norm{\vartheta^{k}-\vartheta^{k-1}}_{\mathcal{H}}^{2}
\\ \!\!\!&&\!\!\!+ 2(1-\alpha_{k})\frac{\lambda_{k}}{\alpha_{k}}\inner{\vartheta^{k}-\vartheta^{k-1}}{\vartheta^{k}-\vartheta^{\ast}}_{\mathcal{H}} \\
\!\!\!&&\!\!\!  -\frac{1-\alpha_{k}}{\alpha_{k}}\beta_{k}(2-\beta_{k})
\norm{J_{r^{-1}}^{\widehat{T}}(\hat{\vartheta}^{k})-\hat{\vartheta}^{k}}_{\mathcal{H}}^{2}
+2\inner{\vartheta^{0}-\vartheta^{\ast}}{\vartheta^{k+1}-\vartheta^{\ast}}_{\mathcal{H}}\Big]\\
\!\!\!&=&\!\!\!(1-\alpha_{k})\norm{\vartheta^{k}-\vartheta^{\ast}}_{\mathcal{H}}^{2}
+\alpha_{k}\gamma_{k}.
\end{eqnarray*}
This proves \eqref{Inequality4}.

Step 3. In this step, we prove that
$$0\leq\limsup\limits_{k\rightarrow\infty}\gamma_{k}<+\infty.$$

By (\ref{B3'}), we deduce that there exists a constant $M_{2}\geq 0$ such that
\begin{equation*}
\frac{\lambda_{k}}{\alpha_{k}}\norm{\vartheta^{k}-\vartheta^{k-1}}_{\mathcal{H}}\leq M_{2},~\forall~k\geq 0.
\end{equation*}
By $\alpha_{k}\in(0,1)$, $\beta_{k}\in (0,2)$ for any $k\geq 0$, we have $\frac{1-\alpha_{k}}{\alpha_{k}}\beta_{k}(2-\beta_{k})>0$  for all $k\geq 0$. Then, by the boundedness of $\{\vartheta^{k}\}_{k=0}^{\infty}$, we have
\begin{eqnarray*}
\sup\limits_{k\geq 0}\gamma_{k}
\!\!\!&=&\!\!\! \sup\limits_{k\geq 0}\Big\{\alpha_{k}(1-\alpha_{k})
\frac{\lambda_{k}^{2}}{\alpha_{k}^{2}}\norm{\vartheta^{k}-\vartheta^{k-1}}_{\mathcal{H}}^{2}
+ 2(1-\alpha_{k})\frac{\lambda_{k}}{\alpha_{k}}
\inner{\vartheta^{k}-\vartheta^{k-1}}{\vartheta^{k}-\vartheta^{\ast}}_{\mathcal{H}} \\
\!\!\!&&\!\!\!\quad\ \ \ ~~-\frac{1-\alpha_{k}}{\alpha_{k}}\beta_{k}(2-\beta_{k})
\norm{J_{r^{-1}}^{\widehat{T}}(\hat{\vartheta}^{k})-\hat{\vartheta}^{k}}_{\mathcal{H}}^{2}
+2\inner{\vartheta^{0}-\vartheta^{\ast}}{\vartheta^{k+1}-\vartheta^{\ast}}_{\mathcal{H}}\Big\} \\
\!\!\!&\leq&\!\!\! \sup\limits_{k\geq 0}\Big\{\alpha_{k}(1-\alpha_{k})
\frac{\lambda_{k}^{2}}{\alpha_{k}^{2}}\norm{\vartheta^{k}-\vartheta^{k-1}}_{\mathcal{H}}^{2}
+ 2(1-\alpha_{k})\frac{\lambda_{k}}{\alpha_{k}}
\inner{\vartheta^{k}-\vartheta^{k-1}}{\vartheta^{k}-\vartheta^{\ast}}_{\mathcal{H}} \\
\!\!\!&&\!\!\!\quad\quad\ +2\inner{\vartheta^{0}-\vartheta^{\ast}}{\vartheta^{k+1}-\vartheta^{\ast}}_{\mathcal{H}}\Big\} \\
 \!\!\!&\leq&\!\!\! M_{2}^{2}+2M_{2}\big(\sup\limits_{k\geq 0}
\norm{\vartheta^{k}}_{\mathcal{H}}+\norm{\vartheta^{\ast}}_{\mathcal{H}}\big)
+2\norm{\vartheta^{0}-\vartheta^{\ast}}_{\mathcal{H}}\big(\sup\limits_{k\geq 0}\norm{\vartheta^{k}}_{\mathcal{H}}+\norm{\vartheta^{\ast}}_{\mathcal{H}}\big)\\
\!\!\!&<&\!\!\!\infty,
\end{eqnarray*}
which implies that $\limsup\limits_{k\rightarrow\infty}\gamma_{k}<\infty$.

We now show that $\limsup\limits_{k\rightarrow\infty}\gamma_{k}\geq 0$. If $\limsup\limits_{k\rightarrow\infty}\gamma_{k}<0$, then there exists $k_{0}$ and $b>0$ such that $\gamma_{k}\leq -b$ for all $k\geq k_{0}$. It follows from \eqref{Inequality4} that
\begin{equation*}
\begin{split}
\norm{\vartheta^{k+1}-\vartheta^{\ast}}_{\mathcal{H}}^{2}
\leq & (1-\alpha_{k})\norm{\vartheta^{k}-\vartheta^{\ast}}_{\mathcal{H}}^{2}+\alpha_{k}\gamma_{k} \\
\leq & (1-\alpha_{k})\norm{\vartheta^{k}-\vartheta^{\ast}}_{\mathcal{H}}^{2}-b\alpha_{k} \\
\leq & \norm{\vartheta^{k}-\vartheta^{\ast}}_{\mathcal{H}}^{2}-b\alpha_{k}
\end{split}
\end{equation*}
for all $k\geq k_{0}$. Consequently,
\begin{equation*}
0\leq b\alpha_{k} \leq \norm{\vartheta^{k}-\vartheta^{\ast}}_{\mathcal{H}}^{2}- \norm{\vartheta^{k+1}-\vartheta^{\ast}}_{\mathcal{H}}^{2}~~\mbox{for all}~~k\geq k_{0}.
\end{equation*}
By induction, we have
\begin{equation*}
0\leq b\sum\limits_{i=k_{0}}^{k}\alpha_{i} \leq \norm{\vartheta^{k_{0}}-\vartheta^{\ast}}_{\mathcal{H}}^{2}- \norm{\vartheta^{k+1}-\vartheta^{\ast}}_{\mathcal{H}}^{2}\leq \norm{\vartheta^{k_{0}}-\vartheta^{\ast}}_{\mathcal{H}}^{2}.
\end{equation*}
Then, taking $k\rightarrow \infty$, it follows that
\begin{equation*}
0\leq b\sum\limits_{i=k_{0}}^{\infty}\alpha_{i} \leq \norm{\vartheta^{k_{0}}-\vartheta^{\ast}}_{\mathcal{H}}^{2}<\infty,
\end{equation*}
which contradicts to condition (B1) that $\sum\limits_{k=0}^{\infty}\alpha_{k}=\infty$. Therefore,  $\limsup\limits_{k\rightarrow\infty}\gamma_{k}\geq 0$.

Step 4. In this step, we prove that
$$\limsup\limits_{k\rightarrow\infty}\gamma_{k}= 0$$
and complete the proof of the strong  convergence of $\{\vartheta^{k}\}_{k=0}^{\infty}$.

By Step 3, $\limsup\limits_{k\rightarrow\infty}\gamma_{k}$ is a finite number.
Then, we can take a subsequence $\{\gamma_{k_{j}}\}_{j=0}^{\infty}$ from $\{\gamma_{k}\}_{k=0}^{\infty}$ such that
\begin{equation} \label{Inequality12}
\begin{split}
 \limsup\limits_{k\rightarrow\infty}\gamma_{k}
= &\lim\limits_{j\rightarrow\infty}\gamma_{k_{j}} \\
= &\lim\limits_{j\rightarrow\infty}\Big\{\alpha_{k_{j}}(1-\alpha_{k_{j}})
\frac{\lambda_{k_{j}}^{2}}{\alpha_{k_{j}}^{2}}
\norm{\vartheta^{k_{j}}-\vartheta^{k_{j}-1}}_{\mathcal{H}}^{2}+ 2(1-\alpha_{k_{j}})\frac{\lambda_{k_{j}}}{\alpha_{k_{j}}}
\inner{\vartheta^{k_{j}}-\vartheta^{k_{j}-1}}{\vartheta^{k_{j}}-\vartheta^{\ast}}_{\mathcal{H}} \\
& ~~-\frac{1-\alpha_{k_{j}}}{\alpha_{k_{j}}}\beta_{k_{j}}(2-\beta_{k_{j}})
\norm{J_{r^{-1}}^{\widehat{T}}(\hat{\vartheta}^{k_{j}})-\hat{\vartheta}^{k_{j}}}_{\mathcal{H}}^{2}
+2\inner{\vartheta^{0}-\vartheta^{\ast}}{\vartheta^{k_{j}+1}-\vartheta^{\ast}}_{\mathcal{H}}\Big\} \\
=& \lim\limits_{j\rightarrow\infty}
\Big\{2\inner{\vartheta^{0}-\vartheta^{\ast}}{\vartheta^{k_{j}+1}-\vartheta^{\ast}}_{\mathcal{H}}
-\frac{1-\alpha_{k_{j}}}{\alpha_{k_{j}}}\beta_{k_{j}}(2-\beta_{k_{j}})
\norm{J_{r^{-1}}^{\widehat{T}}(\hat{\vartheta}^{k_{j}})-\hat{\vartheta}^{k_{j}}}_{\mathcal{H}}^{2}\Big\},
\end{split}
\end{equation}
where the last equation holds from the facts that $\lim\limits_{k\rightarrow\infty}\frac{\lambda_{k}}{\alpha_{k}}\norm{\vartheta^{k}-\vartheta^{k-1}}_{\mathcal{H}}=0$,  $\{\vartheta^{k}\}_{k=0}^{\infty}$ is bounded and  $\{\alpha_{k}\}_{k=0}^{\infty}\subseteq (0,1)$.

Since the sequence $\{\vartheta^{k}\}_{k=0}^{\infty}$ is bounded, there is a subsequence of $\{\vartheta^{k_{j}+1}\}_{j=0}^{\infty}$ converging weakly to some $\hat \vartheta$. Without loss of generality, we assume that
\begin{equation}\label{eq43}
\vartheta^{k_{j}+1}\rightharpoonup  \hat \vartheta    \text{ as } j\to \infty.
\end{equation}
Then,  we have
\begin{equation*}
\lim\limits_{j\rightarrow\infty}
\inner{\vartheta^{0}-\vartheta^{\ast}}{\vartheta^{k_{j}+1}-\vartheta^{\ast}}_{\mathcal{H}}
=\inner{\vartheta^{0}-\vartheta^{\ast}}{\hat\vartheta-\vartheta^{\ast}}_{\mathcal{H}}.
\end{equation*}
Consequently, it turns out from \eqref{Inequality12} that the limit
\begin{equation} \label{Inequality13}
\lim\limits_{j\rightarrow\infty}\beta_{k_{j}}(2-\beta_{k_{j}})\frac{1-\alpha_{k_{j}}}{\alpha_{k_{j}}}
\norm{J_{r^{-1}}^{\widehat{T}}(\hat{\vartheta}^{k_{j}})-\hat{\vartheta}^{k_{j}}}_{\mathcal{H}}^{2}
\end{equation}
exits.
Notice that $0<\liminf\limits_{k\rightarrow\infty}\beta_{k}\leq\limsup\limits_{k\rightarrow\infty}\beta_{k}<2$, then we have $\liminf\limits_{k\rightarrow\infty}\beta_{k}(2-\beta_{k})>0$. It implies that
\begin{equation*}
\Bigg\{\frac{1-\alpha_{k_{j}}}{\alpha_{k_{j}}}\norm{J_{r^{-1}}^{\widehat{T}}(\hat{\vartheta}^{k_{j}})-\hat{\vartheta}^{k_{j}}}_{\mathcal{H}}^{2}\Bigg\}_{j=0}^{\infty}
\end{equation*}
is a bounded sequence. By the assumption that $\lim\limits_{k\rightarrow\infty}\alpha_{k}=0$ in condition (B1), we deduce that
\begin{equation*}
\lim\limits_{j\rightarrow\infty}\norm{J_{r^{-1}}^{\widehat{T}}(\hat{\vartheta}^{k_{j}})-\hat{\vartheta}^{k_{j}}}_{\mathcal{H}}^{2}
=\lim\limits_{j\rightarrow\infty}\frac{\alpha_{k_{j}}}{1-\alpha_{k_{j}}}\cdot\Big(\frac{1-\alpha_{k_{j}}}{\alpha_{k_{j}}}
\norm{J_{r^{-1}}^{\widehat{T}}(\hat{\vartheta}^{k_{j}})-\hat{\vartheta}^{k_{j}}}_{\mathcal{H}}^{2}\Big)=0.
\end{equation*}
Consequently,
\begin{equation}\label{eq44}
\lim\limits_{j\rightarrow\infty}\norm{J_{r^{-1}}^{\widehat{T}}(\hat{\vartheta}^{k_{j}})-\hat{\vartheta}^{k_{j}}}_{\mathcal{H}}=0.
\end{equation}
Then, by the definition of $\nu^{k}$, we obtain that
\begin{equation}\label{Inequality14+}
\norm{\nu^{k_{j}}-\hat{\vartheta}^{k_{j}}}_{\mathcal{H}}
=\beta_{k_{j}}\norm{J_{r^{-1}}^{\widehat{T}}(\hat{\vartheta}^{k_{j}})-\hat{\vartheta}^{k_{j}}}_{\mathcal{H}}=0~~\mbox{as}~~j\rightarrow\infty.
\end{equation}

By condition (B1), (\ref{B3'}) and the definition of $\hat{\vartheta}^{k}$, we have
\begin{equation} \label{Inequality14}
\norm{\hat{\vartheta}^{k_{j}}-\vartheta^{k_{j}}}_{\mathcal{H}}
=\lambda_{k_{j}}\norm{\vartheta^{k_{j}}-\vartheta^{k_{j}-1}}_{\mathcal{H}}
=\alpha_{k_{j}}\cdot\frac{\lambda_{k_{j}}}{\alpha_{k_{j}}}\norm{\vartheta^{k_{j}}-\vartheta^{k_{j}-1}}_{\mathcal{H}}
=0 \mbox{ as } j\rightarrow\infty.
\end{equation}
This, together with \eqref{Inequality14+}, implies that
\begin{equation}\label{Inequality14++}
\norm{\nu^{k_{j}}-\vartheta^{k_{j}}}_{\mathcal{H}}
\leq\norm{\nu^{k_{j}}-\hat{\vartheta}^{k_{j}}}_{\mathcal{H}}+\norm{\hat{\vartheta}^{k_{j}}-\vartheta^{k_{j}}}_{\mathcal{H}}=0~~\mbox{as}~~j\rightarrow\infty.
\end{equation}
Then, by condition (B1), the definition of $\vartheta^{k+1}$,  the boundedness of $\{\vartheta^{k}\}_{k=0}^{\infty}$ and (\ref{Inequality14++}), we obtain that
\begin{equation} \label{Inequality15}
\begin{split}
\norm{\vartheta^{k_{j}+1}-\vartheta^{k_{j}}}_{\mathcal{H}}
& =\norm{\alpha_{k_{j}}\vartheta^{0}+(1-\alpha_{k_{j}})\nu^{k_{j}}-\vartheta^{k_{j}}}_{\mathcal{H}} \\
& \leq\alpha_{k_{j}}\norm{\vartheta^{0}-\vartheta^{k_{j}}}_{\mathcal{H}}
+(1-\alpha_{k_{j}})\norm{\nu^{k_{j}}-\vartheta^{k_{j}}}_{\mathcal{H}}\rightarrow 0~~\mbox{as}~~j\rightarrow\infty.
\end{split}
\end{equation}

By \eqref{eq44}, \eqref{Inequality14}, \eqref{Inequality15} and (i) in Proposition \ref{Convergence proposition},  we obtain that
\begin{eqnarray*}
\!\!\!&&\!\!\!\lim\limits_{j\rightarrow\infty}\norm{J_{r^{-1}}^{\widehat{T}}(\vartheta^{k_{j}+1})-\vartheta^{k_{j}+1}}_{\mathcal{H}}\nonumber\\
\!\!\!&=& \!\!\! \lim\limits_{j\rightarrow\infty}\norm{J_{r^{-1}}^{\widehat{T}}(\vartheta^{k_{j}+1})-J_{r^{-1}}^{\widehat{T}}(\vartheta^{k_{j}}) +J_{r^{-1}}^{\widehat{T}}(\vartheta^{k_{j}})-J_{r^{-1}}^{\widehat{T}}(\hat{\vartheta}^{k_{j}})
+J_{r^{-1}}^{\widehat{T}}(\hat{\vartheta}^{k_{j}})-\hat{\vartheta}^{k_{j}}\nonumber\\
\!\!\!&&\!\!\!\qquad\ \
+\hat{\vartheta}^{k_{j}}-\vartheta^{k_{j}} +\vartheta^{k_{j}}   -\vartheta^{k_{j}+1}}_{\mathcal{H}}\nonumber\\
\!\!\!&\le&\!\!\! \lim\limits_{j\rightarrow\infty}
\big(2\norm{\vartheta^{k_{j}+1}-\vartheta^{k_{j}}}_{\mathcal{H}}  +2\norm{\vartheta^{k_{j}}-\hat{\vartheta}^{k_{j}}}_{\mathcal{H}}
+\norm{J_{r^{-1}}^{\widehat{T}}(\hat{\vartheta}^{k_{j}})-\hat{\vartheta}^{k_{j}}}_{\mathcal{H}}
\big)\nonumber\\
\!\!\!&=&\!\!\!0.
\end{eqnarray*}
This, together with \eqref{eq43} and Lemma~\ref{weakly-strongly closed}, implies that   $\hat \vartheta\in S$. Then, by \eqref{Inequality12}, we have
\begin{equation*}
\limsup\limits_{k\rightarrow\infty}\gamma_{k}\leq\lim\limits_{j\rightarrow\infty}
2\inner{\vartheta^{0}-\vartheta^{\ast}}{\vartheta^{k_{j}+1}-\vartheta^{\ast}}_{\mathcal{H}}
=2\inner{\vartheta^{0}-\vartheta^{\ast}}{\hat{\vartheta}-\vartheta^{\ast}}_{\mathcal{H}}\leq 0.
\end{equation*}
This, together with the conclusion of Step 3, implies that $\limsup\limits_{k\rightarrow\infty}\gamma_{k}=0.$
Consequently, by \eqref{Inequality4} and Lemma~\ref{Strong convergence}, we obtain that $\lim\limits_{k\rightarrow\infty}\norm{\vartheta^{k}-\vartheta^{\ast}}_{\mathcal{H}}=0$.

This completes the proof of the strong convergence of $\{\vartheta^{k}\}_{k=0}^{\infty}$.

Next, we prove that the sequence $\{u^{k}\}_{k=0}^{\infty}$ generated by the Halpern-type relaxed inertial inexact $\mbox{PPA}$~\eqref{HRIIPPA 4} converges strongly to the vector $u^{\ast}=\Pi_{S}(u^{0})$. Since $u^{0}=\vartheta^{0}$, $u^{\ast}=\vartheta^{\ast}$. It remains to show $\lim\limits_{k\rightarrow \infty}\norm{u^{k}-\vartheta^{k}}_{\mathcal{H}}=0$. By the definitions of $u^{k+1}$, $\vartheta^{k+1}$, and Proposition~\ref{Convergence proposition} (vi), we have
\begin{align} \label{Inequality17}
\begin{split}
& \norm{u^{k+1}-\vartheta^{k+1}}_{\mathcal{H}} \\
= & (1-\alpha_{k})\norm{(1-\beta_{k})(\hat{u}^{k}-\hat{\vartheta}^{k})
+\beta_{k}\big[J_{r^{-1}}^{\widehat{T}}(\hat{u}^{k}-r^{-1}e^{k})
-J_{r^{-1}}^{\widehat{T}}(\hat{\vartheta}^{k})\big]}_{\mathcal{H}} \\
= & (1-\alpha_{k})\norm{(1-\beta_{k})(\hat{u}^{k}-r^{-1}e^{k}-\hat{\vartheta}^{k}) \\
& +\beta_{k}\big[J_{r^{-1}}^{\widehat{T}}(\hat{u}^{k}-r^{-1}e^{k})
-J_{r^{-1}}^{\widehat{T}}(\hat{\vartheta}^{k})\big]+r^{-1}(1-\beta_{k})e^{k}}_{\mathcal{H}} \\
= & (1-\alpha_{k})\norm{(1-\frac{\beta_{k}}{2})(\hat{u}^{k}-r^{-1}e^{k}-\hat{\vartheta}^{k}) \\
& +\frac{\beta_{k}}{2}\big[(2J_{r^{-1}}^{\widehat{T}}-I)(\hat{u}^{k}-r^{-1}e^{k})
-(2J_{r^{-1}}^{\widehat{T}}-I)(\hat{\vartheta}^{k})\big]+r^{-1}(1-\beta_{k})e^{k}}_{\mathcal{H}} \\
\leq & (1-\alpha_{k})(1-\frac{\beta_{k}}{2})
\norm{\hat{u}^{k}-r^{-1}e^{k}-\hat{\vartheta}^{k}}_{\mathcal{H}} \\
& +(1-\alpha_{k})\frac{\beta_{k}}{2}
\norm{(2J_{r^{-1}}^{\widehat{T}}-I)(\hat{u}^{k}-r^{-1}e^{k})-(2J_{r^{-1}}^{\widehat{T}}-I)(\hat{\vartheta}^{k})}_{\mathcal{H}}
\\
& +r^{-1}(1-\alpha_{k})|1-\beta_{k}|\norm{e^{k}}_{\mathcal{H}} \\
\leq & (1-\alpha_{k})\norm{\hat{u}^{k}-r^{-1}e^{k}-\hat{\vartheta}^{k}}_{\mathcal{H}}
+r^{-1}(1-\alpha_{k})|1-\beta_{k}|\norm{e^{k}}_{\mathcal{H}} \\
\leq & (1-\alpha_{k})\norm{\hat{u}^{k}-\hat{\vartheta}^{k}}_{\mathcal{H}}
+2r^{-1}\norm{e^{k}}_{\mathcal{H}}.
\end{split}
\end{align}
In addition, by the definition of $\hat{u}^{k}$ and $\hat{\vartheta}^{k}$, we have
\begin{equation} \label{Inequality18}
\begin{split}
\norm{\hat{u}^{k}-\hat{\vartheta}^{k}}_{\mathcal{H}}
& = \norm{u^{k}+\theta_{k}(u^{k}-u^{k-1})
-\vartheta^{k}-\lambda_{k}(\vartheta^{k}-\vartheta^{k-1})}_{\mathcal{H}} \\
& = \norm{u^{k}-\vartheta^{k}+\theta_{k}(u^{k}-u^{k-1})
-\lambda_{k}(\vartheta^{k}-\vartheta^{k-1})}_{\mathcal{H}} \\
& \leq \norm{u^{k}-\vartheta^{k}}_{\mathcal{H}}+\theta_{k}\norm{u^{k}-u^{k-1}}_{\mathcal{H}}
+\lambda_{k}\norm{\vartheta^{k}-\vartheta^{k-1}}_{\mathcal{H}}.
\end{split}
\end{equation}
Combining  \eqref{Inequality17} with \eqref{Inequality18}, we deduce that
\begin{eqnarray}\label{eq52}
\!\!\!&&\!\!\! \norm{u^{k+1}-\vartheta^{k+1}}_{\mathcal{H}}\nonumber \\
\!\!\!&\leq&\!\!\!  (1-\alpha_{k})\norm{u^{k}-\vartheta^{k}}_{\mathcal{H}}
+(1-\alpha_{k})\theta_{k}\norm{u^{k}-u^{k-1}}_{\mathcal{H}}
+(1-\alpha_{k})\lambda_{k}\norm{\vartheta^{k}-\vartheta^{k-1}}_{\mathcal{H}}
+2r^{-1}\norm{e^{k}}_{\mathcal{H}} \nonumber\\
 \!\!\!&\leq &\!\!\!(1-\alpha_{k})\norm{u^{k}-\vartheta^{k}}_{\mathcal{H}}
+\theta_{k}\norm{u^{k}-u^{k-1}}_{\mathcal{H}}
+\lambda_{k}\norm{\vartheta^{k}-\vartheta^{k-1}}_{\mathcal{H}}
+2r^{-1}\norm{e^{k}}_{\mathcal{H}} \nonumber\\
\!\!\!&= &\!\!\! (1-\alpha_{k})\norm{u^{k}-\vartheta^{k}}_{\mathcal{H}}
+\alpha_{k}\frac{\theta_{k}}{\alpha_{k}}\norm{u^{k}-u^{k-1}}_{\mathcal{H}}
+\alpha_{k}\frac{\lambda_{k}}{\alpha_{k}}\norm{\vartheta^{k}-\vartheta^{k-1}}_{\mathcal{H}}
+2r^{-1}\norm{e^{k}}_{\mathcal{H}}.
\end{eqnarray}
By condition (B3) and (\ref{B3'}),
\begin{equation}\label{eq53}
\lim_{k\to \infty}(\frac{\theta_{k}}{\alpha_{k}}\norm{u^{k}-u^{k-1}}_{\mathcal{H}}+ \frac{\lambda_{k}}{\alpha_{k}}
\norm{\vartheta^{k}-\vartheta^{k-1}}_{\mathcal{H}})=0.
\end{equation}
Then, by \eqref{eq52}, \eqref{eq53}, condition (B4) and Lemma~\ref{Strong convergence}, we obtain that $\lim\limits_{k\rightarrow\infty}\norm{u^{k}-\vartheta^{k}}_{\mathcal{H}}=0$. This completes the proof of Theorem \ref{theorem4}.

\vspace{+1em}
\vspace{+1em}




\vspace{+1em}

\noindent {\bf Data availability} The data that support the findings of this study are available from the corresponding author upon reasonable request.

\vspace{+1em}

\noindent {\bf Code availability} The code that support the findings of this study are available from the corresponding author upon reasonable request.

\section*{Declarations}

\vspace{+1em}

\noindent {\bf Conflict of interest} The authors declare no conflict of interest.

\vspace{+1em}

\noindent {\bf Consent for publication} Not applicable.

\vspace{+1em}

\noindent {\bf Ethics approval} Not applicable.

{
\def\cprime{$'$} \def\cprime{$'$} \def\cprime{$'$}

}

\end{document}